\documentclass[12pt, a4paper]{article}

\usepackage{stmaryrd}
\usepackage{enumerate}
\usepackage{hyperref}
\usepackage{color}
\usepackage{lineno}
\usepackage{graphicx}
\usepackage{ae}
\usepackage{amsmath}
\usepackage{amssymb}
\usepackage{latexsym}
\usepackage{url}
\usepackage{epsfig}
\usepackage{mathrsfs}
\usepackage{amsfonts}
\usepackage{amsthm}
\usepackage{stmaryrd}
\usepackage{cite}
\usepackage{BOONDOX-cal}
\usepackage{indentfirst}
\usepackage{amssymb}
\usepackage{indentfirst}
\numberwithin{equation}{section}
\allowdisplaybreaks  

\def\qed{\hfill$\Box$\vspace{12pt}}

\usepackage{color}
\newtheorem{theorem}{Theorem}[section]

\newtheorem{lemma}[theorem]{Lemma}
\newtheorem{proposition}[theorem]{Proposition}

\newtheorem{definition}[theorem]{Definition}
\newtheorem{hypothesis}[theorem]{Hypothesis}
\newtheorem{remark}[theorem]{Remark}

\title{Comparison principle for stochastic heat equations driven
by $\alpha$-stable white noises}
\author{Yongjin Wang$^{1,2}$, Chengxin Yan$^{1,}$\thanks{Corresponding author.}~~and~Xiaowen Zhou$^3$\\
{\small $^1$School of Mathematical Sciences, Nankai University, Tianjin {\rm 300071}, China}\\
{\small $^2$School of Business, Nankai University, Tianjin {\rm 300071}, China}
\\
{\small $^3$Department of Mathematics and Statistics, Concordia University, Montreal {\rm H3G 1M8}, Canada}
\\
{Email: {yjwang@nankai.edu.cn};
~{cxyan@mail.nankai.edu.cn}};
\\
~{xiaowen.zhou@concordia.ca}}
\date{}

\begin{document}

\maketitle

\begin{abstract}
For a class of non-linear stochastic heat equations driven by $\alpha$-stable white noises for
$\alpha\in(1,2)$ with Lipschitz coefficients,
we first show the existence and pathwise uniqueness of
$L^p$-valued c\`{a}dl\`{a}g solutions to such an equation
for $p\in(\alpha,2]$
by considering a sequence of approximating stochastic heat equations driven by truncated
 $\alpha$-stable white noises obtained by removing the big jumps from the original $\alpha$-stable white noises.
 If the $\alpha$-stable white noise is 
 spectrally one-sided,
under additional monotonicity assumption on
noise coefficients, we prove
a comparison theorem on  the $L^2$-valued c\`{a}dl\`{a}g solutions of such an equation.
As a consequence, the non-negativity of the
$L^2$-valued  c\`{a}dl\`{a}g solution
is established for the above stochastic heat equation with 
non-negative initial function.
\bigskip


\bigskip

\noindent{{\bf MCS Subject Classification (2020):} Primary: 60H15; secondary: 60G52, 35B51.}
\end{abstract}


\section{Introduction}
In this paper we study the comparison
principle for a class of stochastic
heat equations. More precisely,
we want to show that if both the initial functions and the
drift coefficients are ordered, then
the solutions of the
following non-linear stochastic heat equation
\begin{equation}
\label{eq:originalequation1}
\left\{\begin{array}{lcl}
\dfrac{\partial u(t,x)}{\partial t}=\dfrac{1}{2}
\dfrac{\partial^2u(t,x)}{\partial x^2}+f(t,x,u(t,x))
\\[0.3cm]
\quad\quad\quad\quad\,\,\,
+\varphi(t-,x,u(t-,x))\dot{L}_{\alpha}(t,x),
&& (t,x)\in (0,T]
\times(0,L),\\[0.3cm]
u(0,x)=u_0(x),&&x\in[0,L],\\[0.3cm]
u(t,0)=u(t,L)=0,&& t\in[0,T],
\end{array}\right.
\end{equation}
are also ordered. In equation (\ref{eq:originalequation1}), $T>0,L>0$ are arbitrary fixed constants,
$\dot{L}_{\alpha}\equiv\{\dot{L}_{\alpha}(t,x):(t,x)\in[0,T]\times[0,L]\}$
denotes an $\alpha$-stable white
noise on $[0,T]\times[0,L]$
with $\alpha\in(1,2)$, the initial function
$u_0$  can be random,
functions $f:[0,T]\times[0,L]\times\mathbb{R}\rightarrow\mathbb{R}$ and
$\varphi:[0,T]\times[0,L]\times\mathbb{R}\rightarrow\mathbb{R}$ are
the drift coefficient and the noise coefficient, respectively.

For the main results in this paper we need  the following hypothesis:
\begin{hypothesis}
\label{hypotheses}

\begin{enumerate}
\item[(i)] Functions $f(t,x,\tilde{u})$ and 
$\varphi(t,x,\tilde{u})$  in 
equation (\ref{eq:originalequation1})  
are globally Lipschitz continuous 
in $\tilde{u}$,
that is, there exists a constant $C$ such that given any $(t,x)\in[0,\infty)\times[0,L]$,
\begin{align*}
\vert f(t,x,\tilde{u})-f(t,x,\tilde{v}) \vert+\vert \varphi(t,x,\tilde{u})-\varphi(t,x,\tilde{v})\vert\leq C\vert \tilde{u}-\tilde{v} \vert
\end{align*}
for all  $\tilde{u},\tilde{v}\in\mathbb{R}$;
\item[(ii)] $\varphi:\tilde{u}\in\mathbb{R}\mapsto
\varphi(t,x,\tilde{u})\in\mathbb{R}$ is non-decreasing
for all $(t,x)\in[0,\infty)\times[0,L]$.
\end{enumerate}
\end{hypothesis}

The comparison principle for stochastic
partial differential equations
(SPDEs for short) driven by Gaussian (continuous) type noises
has been extensively studied, see, for example,
Chen and Kim \cite{Chen:2017}, Donati-Martin and Pardoux\cite{Donati:1993} and Moreno Flores \cite{Moreno Flores:2014} for Gaussian  space-time white noises;
Chen and Huang \cite{Chen:2019} and Xiong and Yang \cite{Xiong:2020}
for Gaussian colored noises that is white in time and colored in space; Denis et al. \cite{Denis:2009} and  Kotelenez \cite{Kotelenez:1992}
for cylindrical Brownian motions and references therein.

Although the existence and uniqueness
of solutions to SPDEs driven by
L\'{e}vy (discontinuous) type noises
have been well studied, see, for example, Albeverio et al. \cite{Albeverio:1998} for Poisson white noise;
 Bo and Wang \cite{Bo:2006},
Peszat and Zabczyk \cite{Peszat:2006} and Truman and Wu \cite{Truman:2003}
for L\'{e}vy space-time white noise;
Peszat and Zabczyk \cite{Peszat:2007} for infinite dimensional
L\'{e}vy processes;
Wang et al. \cite{Wang:2022} for truncated
$\alpha$-stable white noises;
Balan \cite{Balan:2014}, Mytnik \cite{Mytnik:2002} and 
Yang and Zhou \cite{Yang:2017} for $\alpha$-stable white noises;
Xiong and Yang \cite{Xiong:2019} for $\alpha$-stable colored noises that is white in time
and colored in space
and references therein, there are few investigations
on comparison principle for
SPDEs driven by L\'{e}vy noises.

To the best of our knowledge,
Niu and Xie \cite{Niu:2019} first prove
a comparison principle for random field solutions of stochastic heat equations driven by L\'{e}vy 
space-time white noises,
 and show the non-negativity of the solution for non-negative  initial function.
The approach for showing the comparison principle in \cite{Niu:2019}
 is first  introducing a sequence of
symmetric mollifiers and a sequence of bounded linear operators to smooth the
L\'{e}vy space-time white noise and to approximate the Laplace operator, respectively,
so that one can construct
a sequence of approximating stochastic differential equations with
semi-martingale solutions,
driven by the mollified L\'{e}vy processes,
and then show that the random field solutions of these approximating
stochastic differential equations indeed converge to the random
field solution of the original stochastic heat equation.
Consequently, the comparison principle for the
original stochastic heat equation follows by
 the comparison principle for these approximating stochastic differential equations  applying It\^{o}'s formula.


 For a comparison principle on
 function-valued solutions
 to stochastic integro-differential equations driven by  L\'{e}vy processes,
we refer to Dareiotis and Gy\"{o}ngy \cite{Dareiotis:2014} and references therein.
Note that the above two studies are both based on a crucial hypothesis that the noises are square integrable that excludes the case of
$\alpha$-stable noises
or general heavy-tailed noises, which motivates our consideration in the present paper.

The main contribution of this paper is to study the existence and pathwise uniqueness of function-valued solutions to equation (\ref{eq:originalequation1}) and the associated  comparison principle.
We first show that there exists a pathwise unique strong $L^p([0,L])$-valued
 c\`{a}dl\`{a}g solution to equation
(\ref{eq:originalequation1}) for $p\in(\alpha,2]$
by constructing a sequence of approximating stochastic heat equations driven by truncated
$\alpha$-stable white noises.
More precisely, the approach is
to first solve an equation of the form (\ref{eq:originalequation1}) driven by truncated $\alpha$-stable white noises
obtained by removing all the big jumps of size exceeding
a fixed value $K$ from $\dot{L}_{\alpha}$, which
results in a  pathwise unique strong $L^p([0,L])$-valued
c\`{a}dl\`{a}g solution $u^K_t$  for $p\in(\alpha,2]$,
and then show that any two solutions
$u_t^M$ and $u^K_t$ are consistent.
Such a localization method is similar to that in
Balan \cite{Balan:2014} for showing the
existence of random field solutions
and to that in Peszat and Zabczyk
\cite{Peszat:2006} for showing the existence of weak
Hilbert-space-valued solutions to SPDEs driven by $\alpha$-stable white noises.

We then show the comparison principle for $L^2([0,L])$-valued (Hilbert-space-valued) c\`{a}dl\`{a}g solutions to equation (\ref{eq:originalequation1}) with ordered initial functions and drift coefficients
using the method of convolution approximation.
The key of this approach is  to take 
convolution of the equations driven by the truncated
$\alpha$-stable white noise with  the fundamental
solution of heat equation with
homogeneous Dirichlet boundary condition
to obtain a series of
real value semi-martingale
so that  one can use the classical It\^{o} formula 
to obtain the associated comparison principle.
As a consequence, the comparison principle for equation (\ref{eq:originalequation1}) then follows from the localization method mentioned earlier.
The idea of convolution approximation is also applied in
\cite{Xiong:2020}
in showing a comparison principle
for stochastic heat equations with non-homogeneous
boundary conditions
driven by  Gaussian colored noises that are white in time and colored in space.
Compared to \cite{Xiong:2020}, an additional monotonicity assumption on the noise coefficient
$\varphi$ and more technical issues arise in our setting because of the discontinuous
$\alpha$-stable
white noises with $\alpha\in(1,2)$ in equation (\ref{eq:originalequation1}).

Applying the comparison principle,  we also show that for non-negative initial function there exists a non-negative strong $L^2([0,L])$-valued c\`{a}dl\`{a}g solution to equation (\ref{eq:originalequation1}) under the hypothesis.

Our comparison principle respectively generalizes those in \cite{Donati:1993} and \cite{Niu:2019} in some sense.
However,
our comparison principle is on function-valued c\`{a}dl\`{a}g  solutions, while the random field  solutions are considered in \cite{Donati:1993} and \cite{Niu:2019}.

The rest of this paper is organized as follows. In the next section,  some notation and main results of the existence and  pathwise uniqueness of solutions, the comparison principle, and the non-negativity of the solution to equation (\ref{eq:originalequation1}) are stated.
Section \ref{sec3} contains proofs of the
existence and  pathwise uniqueness of strong $L^p([0,L])$-valued  c\`{a}dl\`{a}g solutions to equation (\ref{eq:originalequation1})
 for $p\in(\alpha,2]$.  Section \ref{sec4} contains proofs of the comparison principle on the strong $L^2([0,L])$-valued c\`{a}dl\`{a}g solutions to equation (\ref{eq:originalequation1}). 
In addition, a  method based on the It\^{o} formula for infinite dimensional semi-martingales 
for proving the comparison principle is briefly summarized
at the end of this section.

\section{Notation and main results}
\label{sec2}
In this section we review some basic facts about
$\alpha$-stable white noises and
stochastic heat equations  mainly to state the notation, and then list the main contributions of this paper.

\subsection{Notation}
\label{sec2.1}
Let $(\Omega, \mathcal{F}, \mathbf{P})$ be a complete probability space equipped with
a filtration
$(\mathcal{F}_t)_{t\geq0}$ satisfying the usual
conditions.
Let $(U_n)$ be a disjoint partition of $\mathbb{R}\setminus\{0\}$ such that
$\nu_{\alpha}(U_n)<\infty$ for each $n$,
where $\nu_{\alpha},\alpha\in(1,2)$ is
the so-called
L\'{e}vy measure given by
\begin{align}
\label{eq:jumpsizemeasure}
\nu_{\alpha}(dz)=(c_{+}z^{-\alpha-1}1_{(0,\infty)}(z)+c_{-}(-z)^{-\alpha-1}1_{(-\infty,0)}(z)
)dz
\end{align}
with $c_{+}+c_{-}=1$. Given $T,L>0$,
let $(\xi_j^{n}),(x_j^{n}),(z_j^{n})$
be independent random variables
 on $(\Omega, \mathcal{F}, \mathbf{P})$
taking values in $[0,T]$, $[0,L]$,
and $\mathbb{R}\setminus\{0\}$, respectively, with their distributions specified by
\begin{align*}
\mathbf{P}\left[\xi_j^{n}>t\right]=
\exp(-Lt\nu_{\alpha}(U_n)), \, t\in[0,T],
\end{align*}
\begin{align*}
\mathbf{P}\left[x_j^{n}\in A\right]=\dfrac{|
A\cap[0,L]|}{L},\quad A\in\mathcal{B}([0,L]),
\end{align*}
and
\begin{align*}
\mathbf{P}\left[z_j^{n}\in
B\right]=\dfrac{\nu_{\alpha}(B\cap U_n)}
{\nu_{\alpha}(U_n)},\quad B\in\mathcal{B}(\mathbb{R}\setminus\{0\}),
\end{align*}
respectively, where
$|\cdot|$ denotes the Lebesgue measure and $\mathcal{B}(\cdot)$ denotes the
Borel $\sigma$-field.

For any $n,j\geq1$ set $\tau_j^{n}=\xi_1^n+\ldots+\xi_j^n$,
then
\begin{align}
\label{Poissonmeasure}
N(dt,dx,dz):=\sum_{n,j\geq1}
\delta_{(\tau_j^{n},x_j^{n},z_j^{n})}(dt,dx,dz)
\end{align}
is a Poisson random measure with the
intensity measure $dtdx\nu_{\alpha}(dz)$,
 where $\delta$ denotes the Dirac
delta distribution.
Given $T>0$,
let $\{L_{\alpha}(t,dx),t\in[0,T]\}$
be a sign-measure valued
process formally defined by
\begin{align*}
L_{\alpha}(t,dx):=\sum_{\tau_j^{n}\leq t}
z_j^{n}\delta_{x_j^{n}}(dx)-tdx
\int_{\mathbb{R}\setminus\{0\}}z\nu_{\alpha}(dz).
\end{align*}
By Peszat and Zabczyk
\cite[Example 7.26]{Peszat:2007},
$L_{\alpha}(t,dx)$ is an $\alpha$-stable
martingale measure and the corresponding distribution-valued derivative
$\{\dot{L}_{\alpha}(t,x):t\in[0,T],x\in[0,L]\}$
is an $\alpha$-stable
white noise.
Moreover, by (\ref{Poissonmeasure}),
\begin{align}
\label{def:stablenoise}
L_{\alpha}(dt,dx)=&\int_{\mathbb{R}\setminus\{0\}}z(N(dt,dx,dz)-dtdx\nu_{\alpha}(dz))
\nonumber\\
=&\int_{\mathbb{R}\setminus\{0\}}z\tilde{N}(dt,dx,dz),
\end{align}
where $\tilde{N}(dt,dx,dz)=N(dt,dx,dz)-dtdx\nu_{\alpha}(dz)$
is the so-called  compensated
Poisson random measure.

Let $\mathcal{G}^{\alpha}$ be the class of almost
surely $\alpha$-integrable functions defined by
\begin{align*}
\mathcal{G}^{\alpha}:=\left\{\mathcal{g}\in\mathbf{B}:
\int_0^tds\int_0^L|\mathcal{g}(s,x)|^{\alpha}dx<\infty,
\mathbf{P}\text{-a.s.}\,\,\text{for all} \,\,t\in[0,T]\right\},
\end{align*}
where $\mathbf{B}$ is the space of progressively
measurable functions on
$[0,T]\times[0,L]\times\Omega$. By \cite[Section 5]{Mytnik:2002}, the
stochastic integral with respect to
$\{L_{\alpha}(ds,dx)\}$
is well defined for all
$\mathcal{g}\in \mathcal{G}^{\alpha}$.

Given  $T>0$ and $p\geq1$,
we denote by $h_t\equiv\{h(t,\cdot),t\in[0,T]\}$
the $L^p([0,L])$-valued
process equipped with  norm
\begin{equation*}
||h_t||_{p}:=\left(\int_0^L|h(t,x)|^pdx\right)^{\frac{1}{p}}.
\end{equation*}
In particular, in the case of $p=2$, we write $H=L^2([0,L])$ for
the Hilbert space with
norm $||\cdot||_H=||\cdot||_2$. 

Let $B([0,L])$ be the space of all Borel functions on $[0,L]$.
For any $b,c\in B([0,L])$ define
\begin{align*}
\langle b,c\rangle:=\int_0^L b(x)c(x)dx
\end{align*} 
if it exists. Let $C([0,L])$ be the space of all continuous functions
on $[0,L]$, and let $C^n([0,L])$ be the subset of $C([0,L])$ with
the bounded continuous derivatives up to the order $n\geq1$.

Throughout this paper, $C$ denotes an arbitrary
positive constant whose value might vary
from line to line.
If $C$ depends on some parameters such as $p$ and $T$,
we denote it by $C_{p,T}$.

Let $G_t(x,y)$ be the fundamental solution of the following
heat equation with homogeneous Dirichlet boundary conditions:
\begin{align}
\label{eq:fund-solu}
\left\{\begin{array}{lcl}
\dfrac{\partial G_t(x,y)}{\partial t}=\dfrac{1}{2}
\dfrac{\partial^2G_t(x,y)}{\partial x^2},
&& t\in (0,T],\,\,x,y\in(0,L),\\[0.3cm]
\lim_{t\downarrow0}G_t(x,y)=\delta_y(x),&&x,y\in[0,L],\\[0.3cm]
G_t(x,0)=G_t(x,L)=0,&& t\in[0,T],x\in[0,L].
\end{array}\right.
\end{align}
Its explicit formula
(see, e.g., Feller \cite[Page 341]{Feller:1971})
is specified by
\begin{equation*}
G_t(x,y)=\dfrac{1}{\sqrt{2\pi t}}\sum_{k=-\infty}^{+\infty}\left\{
\exp\left(-\dfrac{(2kL+x-y)^2}{2t}\right)
-\exp\left(-\dfrac{(2kL-x-y)^2}{2t}
\right)\right\}
\end{equation*}
for $t\in(0,T],x,y\in[0,L]$.
Moreover, it holds that for any $s,t\in[0,T]$ and $x,y,z\in[0,L]$,  
\begin{equation}
\label{eq:Greenetimation0}
\int_0^LG_t(x,y)dy\leq C_T,
\end{equation}
\begin{equation}
\label{eq:Greenetimation1}
\int_0^LG_s(x,y)G_t(y,z)dy=G_{t+s}(x,z),
\end{equation}
\begin{equation}
\label{eq:Greenetimation2}
\int_0^L|G_t(x,y)|^pdy\leq Ct^{-\frac{p-1}{2}},\,\, p\geq1,
\end{equation}
and by setting $G_t^x=G_t(x,\cdot)$ 
and $G_t^y=G_t(\cdot,y)$, it holds that
for any function $h\in L^p([0,L]),p\geq1$,
\begin{align}
\label{eq:convolution}
\lim_{t\downarrow0}||\langle h,G_t^x\rangle-h||_p^p=0,\,\,\,
\lim_{t\downarrow0}||\langle h,G_t^y\rangle-h||_p^p=0,\,\,\,x,y\in[0,L].
\end{align}

\subsection{Main results}
\label{sec2.2}
Stochastic heat
equation (\ref{eq:originalequation1}) is a formal SPDE. 
Given $T,L>0$,
by a solution  to
equation (\ref{eq:originalequation1})
we mean a process $u_t\equiv\{u(t,\cdot), t\in[0,T]\}$ taking 
values from measurable functions on $[0, L]$, adapting to the 
filtration generated by $L_{\alpha}$ and satisfying the 
following weak (variational) form equation:
\begin{align}
\label{eq:variationform}
\langle u_t,\phi\rangle&=\langle u_0,\phi\rangle+\dfrac{1}{2}\int_0^t
\langle u_s,\phi{''}\rangle ds+\int_0^t\langle f(s,\cdot,u_s),\phi\rangle ds
\nonumber\\
&\quad+\int_0^{t+}\int_0^L\varphi(s-,y,u(s-,x))\phi(x)L_{\alpha}(ds,dx)
\nonumber\\
&=\langle u_0,\phi\rangle+\dfrac{1}{2}\int_0^t
\langle u_s,\phi{''}\rangle ds+\int_0^t\langle f(s,\cdot,u_s),\phi\rangle ds
\nonumber\\
&\quad+\int_0^{t+}\int_0^L\int_{\mathbb{R}\setminus\{0\}}
\varphi(s-,x,u(s-,x))\phi(x)z\tilde{N}(ds,dx,dz)
\end{align}
for all $t\in[0,T]$ and for any $\phi\in C^{2}([0,L])$
with $\phi(0)=\phi(L)=\phi^{'}(0)=\phi^{'}(L)=0$ or equivalently
satisfying the
following mild form equation:
\begin{align}
\label{eq:mildform}
u(t,x)=&\int_0^LG_t(x,y)u_0(y)dy
+\int_0^{t}\int_0^L
G_{t-s}(x,y)f(s,y,u(s,y))dsdy
\nonumber\\
&\quad+\int_0^{t+}\int_0^L
G_{t-s}(x,y)
\varphi(s-,y,u(s-,y))
L_{\alpha}(ds,dy)
\nonumber\\
=&\int_0^LG_t(x,y)u_0(y)dy
+\int_0^{t}\int_0^L
G_{t-s}(x,y)f(s,y,u(s,y))dsdy
\nonumber\\
&\quad+\int_0^{t+}\int_0^L
\int_{\mathbb{R}\setminus\{0\}}G_{t-s}(x,y)
\varphi(s-,y,u(s-,y))z\tilde{N}(ds,dy,dz)
\end{align}
for all $t\in [0, T]$ and for a.e. $x\in [0, L]$,
where the second equality in
(\ref{eq:variationform})
or
(\ref{eq:mildform}) follows from
(\ref{def:stablenoise}).
For the equivalence between the weak form (\ref{eq:variationform}) and mild form (\ref{eq:mildform}),
we refer to Walsh \cite{Walsh:1986} and references therein.

The definition
of a strong $L^p([0,L])$-valued  solution
to equation (\ref{eq:originalequation1}) for some $p\geq1$
is given as follows.
\begin{definition}
\label{def:solution}
Stochastic heat equation (\ref{eq:originalequation1})
has a strong $L^p([0,L])$-valued c\`{a}dl\`{a}g
solution  with initial function $u_0$ for some $p\geq1$
if for a given $\alpha$-stable
martingale measure
$L_{\alpha}$ there exists an $L^p([0,L])$-valued
c\`{a}dl\`{a}g process $u_t\equiv\{u(t,\cdot),t\in[0,T]\}$
adapted to the filtration generated
by $L_{\alpha}$ such that
either
equation (\ref{eq:variationform})
or equation (\ref{eq:mildform}) holds.
\end{definition}


We now state the main results in this paper.
The existence and pathwise uniqueness of strong $L^p([0,L])$-valued solutions
to equation (\ref{eq:originalequation1}) for
$p\in(\alpha,2]$ is first given by the following theorem.

\begin{theorem}(Existence and uniqueness)
\label{th:existence}
Given $T>0$,
if the initial function $u_0$ satisfies
$\mathbf{E}[||u_0||_p^p]<\infty$ 
for some $p\in(\alpha,2]$,
then  under  {\rm\textbf{Hypothesis}} {\rm\ref{hypotheses} \rm(i)}
there exists a  pathwise unique strong
$L^p([0,L])$-valued c\`{a}dl\`{a}g solution
$u_t\equiv\{u(t,\cdot),t\in[0,T]\}$
to equation (\ref{eq:originalequation1})
and exists a sequence of stopping times $(R_K)_{K\geq1}$ with $R_K\uparrow+\infty$ $\mathbf{P}$-a.s. as $K\uparrow+\infty$
such that for any $K\geq1$
\begin{equation}
\label{eq:moment}
\sup_{0\leq t\leq T}\mathbf{E}\left[||u_t1_{\{t< R_K\}}||_p^p\right]<\infty.
\end{equation}
\end{theorem}

The sequence of stopping times $(R_K)_{K\geq1}$  in
Theorem \ref{th:existence} is defined in
(\ref{eq:stoppingtimes}) with respect to the filtrations generated by the truncated
$\alpha$-stable white noise
$(\dot{L}_{\alpha}^K)_{K\geq1}$ defined in
(\ref{eq:trancatednoise}). The
proof of Theorem \ref{th:existence} is deferred to Section \ref{sec3} below.

To show the comparison theorem for equation (\ref{eq:originalequation1}) with different initial functions and drift coefficients, we consider $H$-valued solutions
 to equation (\ref{eq:originalequation1}).
For this, we further assume that
$\mathbf{E}[||u_0||_H^2]<\infty$.
By Theorem \ref{th:existence} with $p=2$,
there exists a pathwise
unique strong $H$-valued c\`{a}dl\`{a}g solution
$u_t\equiv\{u(t,\cdot),t\in[0,T]\}$
satisfying
\begin{equation*}
\label{eq:momentu}
\sup_{0\leq t\leq T}\mathbf{E}\left[||u_t1_{\{t<R_K\}}||_H^2\right]<\infty.
\end{equation*}

To present the comparison principle, we further consider the following
non-linear stochastic heat equation driven by the same $\alpha$-stable
white noise $\dot{L}_{\alpha}$ with the same noise coefficient $\varphi$ as in equation (\ref{eq:originalequation1}),
that is,
\begin{equation}
\label{eq:originalequation2}
\left\{\begin{array}{lcl}
\dfrac{\partial v(t,x)}{\partial t}=\dfrac{1}{2}
\dfrac{\partial^2v(t,x)}{\partial x^2}+g(t,x,v(t,x))
\\[0.3cm]
\quad\quad\quad\quad\,\,\,
+\varphi(t-,x,v(t-,x))\dot{L}_{\alpha}(t,x),
&& (t,x)\in (0,T]
\times(0,L),\\[0.3cm]
v(0,x)=v_0(x),&&x\in[0,L],\\[0.3cm]
v(t,0)=v(t,L)=0,&& t\in[0,T],
\end{array}\right.
\end{equation}
where the initial function $v_0$
is different from $u_0$ in equation (\ref{eq:originalequation1})
and satisfies $\mathbf{E}[||v_0||_H^2]<\infty$,
the drift coefficient
 $g:[0,T]\times[0,L]\times\mathbb{R}\rightarrow\mathbb{R}$
and
noise coefficient
$\varphi:[0,T]\times[0,L]\times\mathbb{R}\rightarrow\mathbb{R}$ also
satisfy {\rm\textbf{Hypothesis}} \ref{hypotheses}.
Then by Theorem \ref{th:existence} with $p=2$,
there exists a pathwise unique
strong $H$-valued c\`{a}dl\`{a}g solution
$v_t\equiv\{v(t,\cdot),t\in[0,T]\}$
to equation (\ref{eq:originalequation2})
such that for any $K\geq1$
\begin{equation*}
\label{eq:momentv}
\sup_{0\leq t\leq T}\mathbf{E}\left[||v_t1_{\{t<R_K\}}||_H^2\right]<\infty
\end{equation*}
for the same
stopping times $(R_K)_{K\geq1}$ as in Theorem \ref{th:existence}.

We next state the comparison theorem on the
strong $H$-valued c\`{a}dl\`{a}g solutions $u_t$ and $v_t$
 to equations (\ref{eq:originalequation1}) and (\ref{eq:originalequation2}), respectively.

\begin{theorem}(Comparison principle)
\label{th:comparison}
Suppose that  {\rm\textbf{Hypothesis}} {\rm\ref{hypotheses}} holds and $c_{-}=0$ in (\ref{eq:jumpsizemeasure}).
Let
$u_t\equiv\{u(t,\cdot),t\in[0,T]\}$ and
$v_t\equiv\{v(t,\cdot),t\in[0,T]\}$ be the strong
$H$-valued c\`{a}dl\`{a}g solutions
to  equations (\ref{eq:originalequation1}) and (\ref{eq:originalequation2}),
respectively. 
If
$u_0(x)\leq v_0(x)\,\, for\,\text{a.e.}\,\,x\in[0,L]$
and
		$f(t,x,y)\leq g(t,x,y)\,\,\text{for all}\,\,(t,x,y)\in[0,T]\times[0,L]\times\mathbb{R}$,
 we have
\begin{align*}
\mathbf{P}\left[u_t\leq v_t\,\,\,\text{for all}\,\,t\in[0,T]\right]=1.
\end{align*}
\end{theorem}

The precise proof of Theorem \ref{th:comparison} is given in Section
\ref{sec4} below.
Based on Theorem \ref{th:comparison}, one can easily
obtain that with non-negative initial function in $H$, there is
a non-negative $H$-valued
c\`{a}dl\`{a}g solution to
equation (\ref{eq:originalequation1}). More precisely,
the following theorem holds.

\begin{theorem}(Non-negative solution)
\label{th:nonnegativity}
Suppose that {\rm\textbf{Hypothesis}} {\rm\ref{hypotheses}} holds
	and $c_{-}=0$ in (\ref{eq:jumpsizemeasure}).  
Let
$u_t\equiv\{u(t,\cdot),t\in[0,T]\}$ be the strong $H$-valued
c\`{a}dl\`{a}g solution to equation (\ref{eq:originalequation1}).
If
$u_0(x)\geq 0\,\,\text{ for a.e.}\,\,x\in[0,L]$
and
$f(t,x,0)=\varphi(t,x,0)=0\,\,
\text{for all}\,\,(t,x)\in[0,T]\times[0,L]$,
we have
\begin{align*}
\mathbf{P}\left[u_t\geq0\,\,\text{for all}\,\,t\in[0,T]\right]=1.
\end{align*}
\end{theorem}
\begin{proof}
Let $u_0(x)\equiv0$  for a.e. $x\in[0,L]$. Then  by the assumptions of
Theorem \ref{th:nonnegativity} we have that
$u_t\equiv0$
is the unique strong $H$-valued c\`{a}dl\`{a}g solution to equation
(\ref{eq:originalequation1}). Therefore,
the non-negativity of the solution $u_t$ to equation
(\ref{eq:originalequation1}) for general non-negative initial function $u_0$ follows from
Theorem \ref{th:comparison}.
\qed
\end{proof}

\begin{remark}
Note that
Theorems \ref{th:comparison} and \ref{th:nonnegativity}
can also be established under 
 {\rm Hypothesis} {\rm\ref{hypotheses}} for which
 $\varphi:\tilde{u}\in\mathbb{R}\mapsto
\varphi(t,x,\tilde{u})\in\mathbb{R}$ is non-increasing
for all $(t,x)\in[0,\infty)\times[0,L]$
and assumption $c_{+}=0$ in (\ref{eq:jumpsizemeasure});
see the crucial estimate (\ref{eq:crucial-equality}) in proof
to Theorem \ref{th:comparison} in Section \ref{sec4}
for more details.
\end{remark}

\section{Proof of Theorem \ref{th:existence}}
\label{sec3}

The proof of Theorem \ref{th:existence}
proceeds in the following three steps.
Given $T>0$, we first construct a sequence
of truncated $\alpha$-stable white noises
$(\dot{L}_{\alpha}^K)_{K\geq1}$  in (\ref{eq:trancatednoise}) and
stopping times $(R_K)_{K\geq1}$ increasing to infinity in (\ref{eq:stoppingtimes}) with respect to
the filtrations generated by $(\dot{L}_{\alpha}^K)_{K\geq1}$;
see Lemma \ref{lem:stopping}. We then construct
a sequence of stochastic heat equations
(\ref{truncated-eq}) driven
by the truncated $\alpha$-stable white noises
$(\dot{L}_{\alpha}^K)_{K\geq1}$ and prove that
for any fixed $K\geq1$ there exists
a pathwise unique strong $L^p([0,L])$-valued
c\`{a}dl\`{a}g solution
$u^K_t\equiv\{u^K(t,\cdot),t\in[0,T]\}$
 to the stochastic heat equation for
$p\in(\alpha,2]$ by using the
Banach fixed point principle; see Proposition \ref{prop:existence}. 
 Finally, we proceed to show the consistency of solutions to
equations (\ref{truncated-eq}) for different $K$,
that is, it holds for any $M\geq K\geq1$ that
 $u^M_t=u^K_t$
for all $t\in[0,R_K)$; see Lemma \ref{lem:consistency}, and prove that $u_t:=u^K_t$ is indeed a pathwise unique solution of equation (\ref{eq:originalequation1}) for
$t\in[0,R_K)$. Therefore, we obtain a pathwise unique solution $u_t\equiv\{u(t,\cdot),t\in[0,T]\}$ of equation (\ref{eq:originalequation1}) by letting $K\uparrow\infty$.

We now construct a sequence
of truncated $\alpha$-stable white noise
$(\dot{L}_{\alpha}^K)_{K\geq1}$ and
stopping times $(R_K)_{K\geq1}$.
For each $K\geq1$ and $T>0$, recall Section \ref{sec2.1} and define
\begin{align*}
Y(t):=\sum_{\tau_j^{n}\leq t}
z_j^{n},\quad t\geq0,
\end{align*}
\begin{align}
\label{eq:stoppingtimes}
R_K:=\inf\{t: |Y(t)-Y(t-)|> K\},
\end{align}
\begin{align}
\label{eq:trun-jumpsize}
\nu_{\alpha}^K(dz):=(c_{+}z^{-\alpha-1}1_{(0,K]}(z)+c_{-}(-z)^{-\alpha-1}1_{[-K,0)}(z)
)dz,\,
1<\alpha<2,\, c_{+}+c_{-}=1,
\end{align}
\begin{align}
\label{eq:trancatednoise}
L_{\alpha}^K(t,dx)&:=\sum_{\tau_j^{n}\leq
t,z_j^{n}\leq K}z_j^{n}
\delta_{x_j^{n}}(dx)-tdx\int_{\mathbb{R}\setminus\{0\}}
z\nu_{\alpha}^K(dz),
\end{align}
and
\begin{align}
\label{eq:Gamma}
\Gamma_{K,T}:=\left\{\omega\in \Omega: z_j^{n}\leq K\,\, \text{for all}\,\, n,j \,\,\text{for which}\,\,
\tau_j^{n}\leq T,
x_j^{n}\in[0,L]\right\}.
\end{align}

From the above definitions, it easy to see that
$(R_K)_{K\geq1}$ is a sequence of stopping
times with
respect to the filtration generated by $L_{\alpha}^K$ and $R_K\leq R_M$ for any $K\leq M$,
 and that for each $K\geq1$ and $T>0$
\begin{align}
\label{eq:stop-eq}
\{\omega:R_K> T\}=\Gamma_{K,T},
\end{align}
and that
\begin{align}
\label{eq:noise-equ}
L_{\alpha}^K(t,dx)(\omega)=L_{\alpha}(t,dx)
(\omega),\,\,\text{on}\,\,(t,\omega)\in[0,T]\times \Gamma_{K,T}.
\end{align}

Similar to
\cite[Section 19.5.1]{Peszat:2006},
 one can show that the stopping times sequence
$(R_K)_{K\geq1}$
converges to infinity with probability one as
$K$ tends to infinity in the following lemma.
\begin{lemma}
\label{lem:stopping}
For each $K\geq1$ and $T>0$, we have
\begin{align*}
\mathbf{P}[R_K>T]=\exp\left(-\dfrac{TLK^{-\alpha}}{\alpha}\right),
\end{align*}
and
\begin{align}
\label{eq:stop-infinity}
\lim_{K\rightarrow+\infty}
\mathbf{P}[R_K>T]=1\,\,\text{and}\,\,
\lim_{K\rightarrow+\infty}R_K=\infty\,\,\mathbf{P}-\text{a.s.}
\end{align}
\end{lemma}

With the truncated $\alpha$-stable white noises
$(\dot{L}_{\alpha}^K)_{K\geq1}$ in hand,
for each fixed $K\geq1$, we construct the approximating stochastic heat equation
for equation (\ref{eq:originalequation1})
as follows:
\begin{align}
\label{truncated-eq}
\left\{\begin{array}{lcl}
\dfrac{\partial u^K(t,x)}{\partial t}=\dfrac{1}{2}
\dfrac{\partial^2u^K(t,x)}{\partial x^2}+f(t,x,u^K(t,x))
\\[0.3cm]
\quad\quad\quad\quad\quad\,\,
+\varphi(t-,x,u^K(t-,x))\dot{L}^K_{\alpha}(t,x),
&& (t,x)\in (0,T]
\times(0,L),\\[0.3cm]
u^K(0,x)=u_0(x),&&x\in[0,L],\\[0.3cm]
u^K(t,0)=u^K(t,L)=0,&& t\in[0,T],
\end{array}\right.
\end{align}
where the drift coefficient $f$, the noise coefficient $\varphi$ and the initial function
$u_0$ are the same as in equation
(\ref{eq:originalequation1}).

Similar to (\ref{Poissonmeasure}), for
each $K\geq1$,  define
\begin{align*}
 N^K(dt,dx,dz):=\sum_{{z_j^n\leq K}}
\delta_{(\tau_j^{n},x_j^{n},z_j^{n})}(dt,dx,dz).
\end{align*}
Then $N^K(dt,dx,dz)$ is a Poisson random measure
with  intensity measure $dtdx\nu_{\alpha}^K(dz)$, 
where $\nu_{\alpha}^K$ is given by
(\ref{eq:trun-jumpsize}). By (\ref{eq:trancatednoise}), one can show
that
\begin{align*}
L_{\alpha}^K(dt,dx)
&=\int_{\mathbb{R}\setminus\{0\}}z(N^K(dt,dx,dz)-dtdx\nu_{\alpha}^K(dz))\\
&=\int_{\mathbb{R}\setminus\{0\}}z\tilde{N}^K(dt,dx,dz),
\end{align*}
where $\tilde{N}^K(dt,dx,dz)=N^K(dt,dx,dz)-dtdx\nu_{\alpha}^K(dz)$ is the compensated
Poisson random measure corresponding to the truncated
$\alpha$-stable measure
$L_{\alpha}^K$.

Given $K\geq1$, by a solution to
equation (\ref{truncated-eq})
we mean a process $u_t^K\equiv\{u^K(t,\cdot), t\in[0,T]\}$ taking 
values from measurable functions on $[0,L]$, adapting to the 
filtration generated by $L_{\alpha}^K$ and satisfying the following
weak form equation:
\begin{align}
\label{eq:weak-truncted}
\langle u_t^K,\phi\rangle
&=\langle u_0,\phi\rangle+\dfrac{1}{2}\int_0^t
\langle u_s^K,\phi{''}\rangle ds+\int_0^t
\langle f(s,\cdot,u_s^K),\phi\rangle ds
\nonumber\\
&\quad+\int_0^{t+}\int_0^L\int_{\mathbb{R}\setminus\{0\}}
\varphi(s-,x,u^K(s-,x))\phi(x)z\tilde{N}^K(ds,dx,dz)
\end{align}
for all $t\in[0,T]$ and for any $\phi\in C^{2}([0,L])$
with $\phi(0)=\phi(L)=\phi^{'}(0)=\phi^{'}(L)=0$ or equivalently
satisfying the
following mild form equation:
\begin{align}
\label{eq:mild-truncted}
u^K(t,x)&=\int_0^LG_t(x,y)u_0(y)dy+\int_0^t\int_0^LG_{t-s}(x,y)f(s,y,u^K(s,y))dsdy
\nonumber\\
&\quad+\int_0^{t+}\int_0^L
\int_{\mathbb{R}\setminus\{0\}}G_{t-s}(x,y)
\varphi(s-,y,u^K(s-,y))z\tilde{N}^K(ds,dy,dz)
\end{align}
for all $t\in [0, T]$ and  for a.e. $x\in [0, L]$.

According to Definition \ref{def:solution},
 stochastic heat equation (\ref{truncated-eq})
has a strong $L^p([0,L])$-valued c\`{a}dl\`{a}g solution with initial function $u_0$ for some $p\geq1$,
if for a given truncated $\alpha$-stable
martingale measure
$L_{\alpha}^K$ defined by (\ref{eq:trancatednoise})
there exists a $L^p([0,L])$-valued
c\`{a}dl\`{a}g process $u_t^K\equiv\{u^K(t,\cdot), t\in[0,T]\}$
such that either equation (\ref{eq:weak-truncted}) or
equation (\ref{eq:mild-truncted}) holds.

We now show that there exists a pathwise
unique strong $L^p([0,L])$-valued
c\`{a}dl\`{a}g solution $u_t^K\equiv\{u^K(t,\cdot), t\in[0,T]\}$
to equation
(\ref{truncated-eq}) for $p\in(\alpha,2]$ 
by using Banach fixed point
principle in the following proposition. The same method was also
applied in \cite{Truman:2003} for showing the existence and uniqueness of 
solution to stochastic Burgers equation driven by L\'{e}vy space-time white
noise.

\begin{proposition}
\label{prop:existence}
Given $K\geq1$ and $T>0$,
If the initial function $u_0$ satisfies
$\mathbf{E}[||u_0||_p^p]<\infty$ 
for some $p\in(\alpha,2]$,
then under 
{\rm\textbf{Hypothesis}} {\rm\ref{hypotheses}} {\rm(i)}
there exists a pathwise unique strong
$L^p([0,L])$-valued c\`{a}dl\`{a}g solution
$u_t^K\equiv\{u^K(t,\cdot),t\in[0,T]\}$
to equation (\ref{truncated-eq})
and
\begin{equation}
\label{eq:truncated-moment}
\sup_{0\leq t\leq T}\mathbf{E}\left[||u^K_t||_p^p\right]<\infty.
\end{equation}
\end{proposition}

\begin{proof}
Let $\mathcal{H}_{T}$ be the space of all
$L^p([0,L])$-valued and
$\mathcal{F}_t$-adapted
c\`{a}dl\`{a}g processes $h_t\equiv\{h(t,\cdot),t\in[0,T]\}$.
The norm
$||\cdot||_{\mathcal{H}_T}$ of the space
$\mathcal{H}_{T}$ is
defined by
\begin{equation}
\label{eq:normH}
||h||_{\mathcal{H}_T}:=
\left(\sup_{0\leq t\leq T}
\mathbf{E}[||h_t||_p^p]
\right)^{\frac{1}{p}},\,\,p\in(\alpha,2].
\end{equation}
Then under the above norm,
$\mathcal{H}_T$ is a Banach space; see, e.g., \cite[Page 236]{Bo:2006}.

Let
$(J^K)_{K\geq0}$
be the operators defined by
\begin{equation*}
J^0(h)(t,x):=\int_0^LG_{t}(x,y)h_0(y)dy,
\end{equation*}
\begin{align*}
J^K(h)(t,x)&:=\int_0^t\int_0^L
G_{t-s}(x,y)f(s,y,h(s,y))dsdy
\\
&\quad+\int_0^{t+}\int_0^L
\int_{\mathbb{R}\setminus\{0\}}G_{t-s}(x,y)\varphi
(s-,y,h(s-,y))z\tilde{N}^K(ds,dy,dz),\,\, K\geq1,
\end{align*}
for all $h\in\mathcal{H}_T$, 
where $(t,x)\in[0,T]\times {[0,L]}$
and $h_0$ satisfies that
$\mathbf{E}[||h_0||_p^p]<\infty$.

We first consider the case of $t\in[0,t_1]$,
where $0<t_1\leq T$ is sufficiently small.

\textbf{Step (i):} Prove that for each
$h\in{\mathcal{H}_{t_1}}$, $J^K(h)\in{\mathcal{H}_{t_1}}, K\geq0$.

By Jensen's inequality
and (\ref{eq:Greenetimation0}), we have
\begin{align*}
\mathbf{E}[|| J^0(h)(t,\cdot)||_p^p]
&=\mathbf{E}\left[\int_0^L
\Bigg |\int_0^LG_{t}
(x,y)h_0(y)dy\Bigg
|^pdx\right]\\
&\leq \mathbf{E}\left[\int_0^L\left(
\int_0^LG_{t}(x,y)dx\right)
|h_0(y)|^pdy\right]\\
&\leq C_T\mathbf{E}[||h_0||_p^p].
\end{align*}
Since  $\mathbf{E}[||h_0||_p^p]<\infty$,
then $J^0(h)\in{\mathcal{H}_{t_1}}$.

For $J^K(h),K\geq1$, the H\"{o}lder inequality and the Burkholder-Davis-Gundy
inequality imply that
\begin{align*}
&\mathbf{E}[||J^K(h)(t,\cdot)||_p^p]
\\
&\leq C_p\int_0^L\mathbf{E}\left[\left|
\int_0^{t}\int_0^L
G_{t-s}(x,y)f(s,y,h(s,y))dsdy
\right|^p\right]dx
\\
&\quad+C_p\int_0^L\mathbf{E}\left[\left|
\int_0^{t+}\int_0^L\int_{\mathbb{R}\setminus\{0\}}
G_{t-s}(x,y)\varphi(s-,y,h(s-,y))z
\tilde{N}^K(ds,dy,dz)\right|^p\right]dx
\\
&\leq C_{p,t_1}\int_0^L\mathbf{E}\left[
\int_0^{t}\int_0^L
|G_{t-s}(x,y)f(s,y,h(s,y))|^pdsdy\right]dx
\\
&\quad+C_p\int_0^L
\mathbf{E}\left[\left|\int_0^{t+}
\int_0^L\int_{\mathbb{R}\setminus\{0\}}
|G_{t-s}(x,y)\varphi(s-,y,h(s-,y))z
|^2N^K(ds,dy,dz)\right|^{\frac{p}{2}}
\right]dx
\\
&\leq C_{p,t_1}\int_0^L\mathbf{E}\left[
\int_0^{t}\int_0^L
|G_{t-s}(x,y)f(s,y,h(s,y))|^pdsdy\right]dx
\\
&\quad+C_p
\int_0^L\mathbf{E}\left[\int_0^{t+}
\int_0^L\int_{\mathbb{R}\setminus\{0\}}|G_{t-s}(x,y)\varphi(s-,y,
h(s,y))z|^pN^K(ds,dy,dz)\right]dx
\\
&=C_{p,t_1}\int_0^L\mathbf{E}\left[
\int_0^{t}\int_0^L
|G_{t-s}(x,y)f(s,y,h(s,y))|^pdsdy\right]dx
\\
&\quad+C_p\int_0^L\mathbf{E}\left[\int_0^{t}
\int_0^L\int_{\mathbb{R}\setminus\{0\}}|G_{t-s}(x,y)\varphi(s,y,
h(s,y))z|^pdsdy\nu_{\alpha}^K(dz)\right]dx
\end{align*}
where the third inequality follows from the
fact that
\begin{align}
\label{eq:element-inequ}
\left|\sum_{i=1}^ka_i^2\right|^{\frac{q}{2}}\leq
\sum_{i=1}^{k}|a_i|^q
\end{align}
for $a_i\in\mathbb{R}, k\geq1$,
and $q\in(0,2]$.

By (\ref{eq:trun-jumpsize}), it holds
for $p\in(\alpha,2]$ that
\begin{align}
\label{eq:jumpestimate}
\int_{\mathbb{R}\setminus\{0\}}\vert z\vert ^p\nu_{\alpha}^K
(dz)=c_{+}\int_0^Kz^{p-\alpha-1}dz+c_{-}\int_{-K}^0(-z)^{p-\alpha-1}dz
=\frac{K^{p-\alpha}}{p-\alpha}.
\end{align}
By 
 {\rm\textbf{Hypothesis}} \ref{hypotheses} (i)  and
(\ref{eq:Greenetimation2}), there exists a constant $C_{p,t_1,\alpha,K}$ such that
\begin{align*}
&\mathbf{E}[||J^K(h)(t,\cdot)||_p^p]
\\
&\leq C_{p,t_1,\alpha,K}
\int_0^L\mathbf{E}\left[
\int_0^{t}\int_0^L
|G_{t-s}(x,y)|^p(|f(s,y,h(s,y))|^p+
|\varphi(s,y,h(s,y))|^p)dsdy\right]dx
\\
&\leq C_{p,t_1,\alpha,K}\int_0^L\int_0^{t}\int_0^L
\mathbf{E}[1+|h(s,y)|^p]|G_{t-s}(x,y)|^pdydsdx
\\
&\leq C_{p,t_1,\alpha,K}\left(\int_0^{t}|t-s|^{-\frac{p-1}{2}}ds
+\int_0^{t}
\mathbf{E}[||h_s||_p^p]|t-s|^{-\frac{p-1}{2}}ds\right)
\\
&\leq C_{p,t_1,\alpha,K}\left(\int_0^{t_1}|t_1-s|^{-\frac{p-1}{2}}ds
+||h||_{\mathcal{H}_{t_1}}^p\int_0^{t_1}|t_1-s|^{-\frac{p-1}{2}}ds\right)
\\
&\leq C_{p,t_1,\alpha,K}(1+||h||_{\mathcal{H}_{t_1}}^p),
\end{align*}
where
\begin{align*}
\int_0^{t_1}|t_1-s|^{-\frac{p-1}{2}}ds=\dfrac{2t_1^{\frac{3-p}{2}}}{3-p}<\infty
\end{align*}
for $p\leq2$.
Therefore, $J^K(h)\in{\mathcal{H}_{t_1}}$ for any $K\geq1$.

Let
$(\mathcal{J}^K)_{K\geq1}$
be the operators defined by
\begin{equation*}
\mathcal{J}^K(h)(t,x):=J^0(h)(t,x)+J^K(h)
(t,x),\,\,K\geq1,
\end{equation*}
for all $h\in{\mathcal{H}_{t_1}}$,
where $(t,x)\in[0,T]\times {[0,L]}$.

\textbf{Step (ii):}
Prove that for each $K\geq1$ the operator
$\mathcal{J}^K$
is a contraction on ${\mathcal{H}_{t_1}}$, i.e., for any 
$h,\bar{h}\in{\mathcal{H}_{t_1}}$
with $h_0=\bar{h}_0$, there exists a constant
$\kappa\in(0,1)$ such that
\begin{align*}
||\mathcal{J}^K(h)-\mathcal{J}^K(\bar{h})||
_{{\mathcal{H}_{t_1}}}
\leq\kappa||h-\bar{h}||_{{\mathcal{H}_{t_1}}}.
\end{align*}

Clearly $\mathcal{J}^K:{\mathcal{H}_{t_1}}\rightarrow
{\mathcal{H}_{t_1}}$ for each $K\geq1$.
For any $h,\bar{h}\in{\mathcal{H}_{t_1}}$ with $h_0=\bar{h}_0$,
we have
\begin{align*}
\mathcal{J}^K(h)-\mathcal{J}^K(\bar{h})
=J^K(h)-J^K(\bar{h}),\,\,K\geq1.
\end{align*}

The H\"{o}lder inequality, 
the Burkholder-Davis-Gundy inequality, 
(\ref{eq:element-inequ})-(\ref{eq:jumpestimate}),
 {\rm\textbf{Hypothesis}} \ref{hypotheses} (i) and
(\ref{eq:Greenetimation2}) together 
imply that 
\begin{align*}
&||\mathcal{J}^K(h)-\mathcal{J}^K(\bar{h})||
_{\mathcal{H}_{t_1}}^p
\\
&\leq C_p\sup_{0\leq t\leq t_1}
\int_0^L\mathbf{E}\bigg[\bigg|
\int_0^{t}\int_0^L
G_{t-s}(x,y)[f(s,y,h(s,y))-f(s,y,\bar{h}(s,y))]dsdy\bigg|^p\bigg]dx
\\
&\quad+C_p\sup_{0\leq t\leq t_1}
\int_0^L\mathbf{E}\bigg[\bigg|
\int_0^{t+}\int_0^L\int_{\mathbb{R}\setminus\{0\}}
G_{t-s}(x,y)
\\
&\quad\quad\times[\varphi(s-,y,h(s-,y))-\varphi
(s-,y,\bar{h}(s-,y))]z\tilde{N}(ds,dy,dz)\bigg|^p\bigg]dx
\\
&\leq C_{p,t_1}\sup_{0\leq t\leq t_1}
\int_0^L\mathbf{E}\bigg[
\int_0^{t}\int_0^L
|G_{t-s}(x,y)[f(s,y,h(s,y))-f(s,y,\bar{h}(s,y))]|^pdsdy\bigg]dx
\\
&\quad+C_{p,t_1,\alpha,K}\sup_{0\leq t\leq t_1}
\int_0^L\mathbf{E}\left[\int_0^{t}
\int_0^L
|G_{t-s}(x,y)[\varphi(s,y,h(s,y))-
\varphi(s,y,\bar{h}(s,y))]|^pdyds\right]dx
\nonumber\\
&\leq C_{p,t_1,\alpha,K}\sup_{0\leq
t\leq t_1}\int_0^L\int_0^{t}
\int_0^L|G_{t-s}(x,y)|^p
\mathbf{E}[|h(s,y)-\bar{h}(s,y)|^p]
dydsdx
\\
&\leq C_{p,t_1,\alpha,K}
\sup_{0\leq t\leq t_1}
\int_0^{t}\int_0^L|{t-s}|^{-\frac{p-1}{2}}
\mathbf{E}[|h(s,y)-\bar{h}(s,y)|^p]
dyds
\\
&\leq C_{p,t_1,\alpha,K}
\left(\sup_{0\leq t\leq t_1}\int_0^{t}
|t-s|^{-\frac{p-1}{2}}ds\right)
||h-\bar{h}||_{\mathcal{H}_{t_1}}^p
\\
&=\dfrac{2C_{p,t_1,\alpha,K}t_1^{\frac{3-p}{2}}}{3-p}||h-\bar{h}||_{\mathcal{H}_{t_1}}^p.
\end{align*}
Since $t_1$ is sufficiently small
and $p\in(\alpha,2]$,
then for each fixed $K\geq1$ we can choose
\begin{equation*}
\kappa=\dfrac{2C_{p,t_1,\alpha,K}t_1^{\frac{3-p}{2}}}{3-p}\in(0,1).
\end{equation*}

Based on the results of 
\textbf{Step (i)} and \textbf{Step (ii)}, by using the
Banach fixed point principle
on the set $\{h\in{\mathcal{H}_{t_1}}:
h(0,x)=u_0(x),x\in[0,L]\}$ and the right continuity of sample paths, one can conclude that
for each $K\geq1$  the equation
(\ref{truncated-eq}) has a pathwise
unique strong $L^p([0,L])$-valued c\`{a}dl\`{a}g
solution $u^K_t\in{\mathcal{H}_{t_1}}$
up to time $t_1$ for $p\in(\alpha,2]$.
For $t>t_1$, one can start with
the time $t_1$
and find a sufficiently small
time interval $\Delta t_1$
so that $t_2=t_1+\Delta t_1$.
Then similar to the previous proof,
one can obtain the solution,
 still denoted by $u^K_t$, in
${\mathcal{H}_{t_2}}$ with $t\in[0,t_2]$.
Therefore, the proof of existence and pathwise uniqueness of solutions is completed via
the above successive procedure, and the moment estimate (\ref{eq:truncated-moment}) follows
from 
 (\ref{eq:normH}).
\qed
\end{proof}

\begin{remark}
\label{re:well-defined}
By {\rm\textbf{Hypothesis}} {\rm\ref{hypotheses}} {\rm(i)}
and estimate (\ref{eq:truncated-moment}), 
the (stochastic) integrals 
on the right-hand side of (\ref{eq:mild-truncted})
are well defined.
\end{remark}

From Proposition \ref{prop:existence}, we know that, for each $K\geq1$, there exists a pathwise unique strong $L^p([0,L])$-valued  c\`{a}dl\`{a}g
solution $u^K_t\equiv\{u^K(t,\cdot),t\in[0,T]\}$ to equation (\ref{truncated-eq}) for $p\in(\alpha,2]$.
To find an $L^p([0,L])$-valued c\`{a}dl\`{a}g
solution of equation (\ref{eq:originalequation1}) using
solutions $(u^K_t)_{K\geq1}$, we first prove the consistency of the solutions
$(u^K_t)_{K\geq1}$ in the following lemma.

\begin{lemma}
\label{lem:consistency}
Under the assumptions in Proposition \ref{prop:existence}, for any $1\leq K\leq M$ we have
\begin{align*}
u^K_t=u^M_t,\,
\mathbf{P}\text{-a.s.\,on}\,\, t\in[0,R_K).
\end{align*}
\end{lemma}
\begin{proof}
It suffices to consider the case of $R_K<T$.
By (\ref{eq:mild-truncted}),
\begin{align*}
&(u^K(t,x)-u^M(t,x))1_{\{t< R_{K}\}}
\\
&=1_{\{t< R_{K}\}}\int_0^t\int_0^LG_{t-s}(x,y)[f(s,y,u^K(s,y))-f(s,y,u^M(s,y))]1_{\{s<R_{K}\}}dsdy
\\
&\quad+1_{\{t< R_{K}\}}\int_0^{t+}\int_0^L
\int_{\mathbb{R}\setminus\{0\}}G_{t-s}(x,y)
\\
&\quad\quad\quad\quad\quad\quad\quad\quad\quad
\times[\varphi(s-,y,u^K(s-,y))-\varphi(s-,y,u^M(s-,y))]z1_{\{s< R_{K}\}}\tilde{N}^K(ds,dy,dz).
\end{align*}
Let us define
$$U_t:=\mathbf{E}\left[||(u^K_t-u^M_t)1_{\{t< R_{K}\}}||_p^p\right],\,\,t\in[0,T].$$
Then
by the H\"{o}lder inequality, the Burkholder-Davis-Gundy inequality, (\ref{eq:element-inequ})-(\ref{eq:jumpestimate})
and  {\rm\textbf{Hypothesis}} \ref{hypotheses} (i),
we have
\begin{align*}
U(t)&\leq C_{p,T,\alpha,K}\int_0^L\int_0^{t}
\int_0^L|G_{t-s}(x,y)|^p
\mathbf{E}[|(u^K(s,y)-u^M(s,y))1_{\{s< R_{K}\}}|^p]
dydsdx
\\
&\leq C_{p,T,\alpha,K}\int_0^{t}|t-s|^{\frac{p-1}{2}}
\mathbf{E}[||(u^K_s-u^M_s)1_{\{s< R_{K}\}}||_p^p]ds
\\
&=C_{p,T,\alpha,K}\int_0^{t}|t-s|^{\frac{p-1}{2}}U(s)ds.
\end{align*}
Therefore, the generalized Gronwall Lemma
(see, e.g., Lin \cite[Theorem 1.2]{Lin:2013})
implies that $U(t)\equiv0$,
 which completes the proof.
 \qed
\end{proof}

Similar to \cite[Proposition 29]{Balan:2014}, for any $t\in[0,T]$ and positive constants
$K,M$, let us define
$$\Omega_t:=\bigcap_{K\leq M}\left\{t< R_K,
u^K_t=u^M_t\right\}$$
and
$$\Omega_t^{*}:=\Omega_t\cap\left\{\lim_{K\rightarrow\infty}R_K=\infty\right\}.$$
By Lemmas \ref{lem:stopping}
and \ref{lem:consistency}, one can conclude that
$\mathbf{P}[\Omega_t^{*}]=1$.

We next show that there exists a unique solution
to equation (\ref{eq:originalequation1}) as the
limit of solutions $(u^K_t)_{K\geq1}$ of
equation (\ref{truncated-eq}).
\\

\begin{Tproof}~{\textbf{of Theorem \ref{th:existence}.}}
We first prove that the process 
$u_t\equiv\{u(t,\cdot),t\in[0,T]\}$ defined by
\begin{equation}
\label{eq:solution}
u(t,\cdot,\omega):=
u^K(t,\cdot,\omega), \,\,
\omega\in\Omega_t^{*},t< R_K,
\end{equation}
is a strong $L^p([0,L])$-valued solution
of equation (\ref{eq:originalequation1}) for
 $p\in(\alpha,2]$.
 By Lemma \ref{lem:consistency},  (\ref{eq:solution})
 is well-defined.
From Proposition \ref{prop:existence}, we know that for a given
 $K\geq1$ there exists a unique
$L^p([0,L])$-valued process $u^K_t$ satisfies
equation (\ref{eq:mild-truncted}) for
 $p\in(\alpha,2]$.
Therefore, for each $K\geq1$, it holds by (\ref{eq:solution}) that
\begin{align*}
u(t,x)1_{\{t< R_K\}}&=u^K(t,x)1_{\{t< R_K\}}
\\
&=1_{\{t< R_K\}}\int_0^LG_t(x,y)u_0(y)dy
\\
&\quad+1_{\{t< R_K\}}\int_0^t\int_0^LG_{t-s}(x,y)f(s,y,u^K(s,y))dsdy
\nonumber\\
&\quad+1_{\{t< R_K\}}\int_0^{t+}\int_0^L
\int_{\mathbb{R}\setminus\{0\}}G_{t-s}(x,y)
\varphi(s-,y,u^K(s-,y))z\tilde{N}^K(ds,dy,dz)
\end{align*}
for all $t\in [0, T]$ and  for a.e. $x\in [0, L]$.

By (\ref{eq:Gamma})-(\ref{eq:noise-equ}) and
 \cite[Remark 19.1]{Peszat:2006} or \cite[Proposition C1]{Balan:2014}, we have
\begin{align*}
u(t,x)1_{\{t< R_K\}}&=1_{\{t< R_K\}}\int_0^LG_t(x,y)u_0(y)dy
\\
&\quad+1_{\{t< R_K\}}\int_0^t\int_0^LG_{t-s}(x,y)f(s,y,u^K(s,y))1_{\{s< R_K\}}dsdy
\nonumber\\
&\quad+1_{\{t< R_K\}}\int_0^{t+}\int_0^L
\int_{\mathbb{R}\setminus\{0\}}G_{t-s}(x,y)
\varphi(s-,y,u^K(s-,y))z1_{\{s< R_K\}}\tilde{N}(ds,dy,dz)
\\
&=1_{\{t< R_K\}}\int_0^LG_t(x,y)u_0(y)dy
\\
&\quad+1_{\{t< R_K\}}\int_0^t\int_0^LG_{t-s}(x,y)f(s,y,u(s,y))dsdy
\nonumber\\
&\quad+1_{\{t< R_K\}}\int_0^{t+}\int_0^L
\int_{\mathbb{R}\setminus\{0\}}G_{t-s}(x,y)
\varphi(s-,y,u(s-,y))z\tilde{N}(ds,dy,dz)
\end{align*}
for all $t\in [0,T]$ and  for a.e.
 $x\in [0, L]$,
 and the desired existence of solution follows from (\ref{eq:stop-infinity})  by letting $K\uparrow\infty$.

For the pathwise uniqueness of solutions to
equation (\ref{eq:originalequation1}), without
loss of generality, for each fixed $K\geq1$ we consider the case of
$t< R_K< T$. Suppose that $u_t\equiv\{u(t,\cdot),t\in[0,T]\}$ and $\bar{u}_t\equiv\{\bar{u}(t,\cdot),t\in[0,T]\}$ are two solutions
of equation (\ref{eq:originalequation1})
with initial functions $u_0$ and $\bar{u}_0$, respectively. 
If $u_0(x)=\bar{u}_0(x)$ 
 for a.e. $x\in[0,L]$, 
by (\ref{eq:mildform}) and 
(\ref{eq:Gamma})-(\ref{eq:noise-equ})
 one can show that 
\begin{align*}
&1_{\{t< R_K\}}(u(t,x)-\bar{u}(t,x))
\\
&=1_{\{t< R_K\}}\int_0^{t}\int_0^L
G_{t-s}(x,y)[f(s,y,u(s,y))-f(s,y,\bar{u}(s,y))]1_{\{s< R_K\}}dsdy
\nonumber\\
&\quad+1_{\{t< R_K\}}\int_0^{t+}\int_0^L
\int_{\mathbb{R}\setminus\{0\}}G_{t-s}(x,y)
\\
&\quad\quad\quad\quad\quad\quad\quad
\quad\quad\times
[\varphi(s-,y,u(s-,y))-\varphi(s-,y,\bar{u}(s-,y))]1_{\{s< R_K\}}z\tilde{N}^K(ds,dy,dz).
\end{align*}

By the Burkholder-Davis-Gundy inequality,
(\ref{eq:element-inequ})-(\ref{eq:jumpestimate}),
 {\rm\textbf{Hypothesis}} \ref{hypotheses} (i)
and (\ref{eq:Greenetimation2}), we have 
\begin{align*}
&\mathbf{E}\left[||(u_t-\bar{u}_t)
1_{\{t< R_K\}}||_p^p\right]
\\
&\leq C_{p} \mathbf{E}\left[\int_0^L
\int_0^{t}\int_0^L
|G_{t-s}(x,y)|^p|u(s,y)-\bar{u}(s,y)|^p1_{\{s< R_K\}}dsdydx\right]
\\
&\quad+ C_{p,\alpha,K} \mathbf{E}\left[\int_0^L
\int_0^{t}\int_0^L
|G_{t-s}(x,y)|^p|u(s,y)-\bar{u}(s,y)|^p1_{\{s<R_K\}}dsdydx\right]
\\
&\leq C_{p,\alpha,K,T} \int_0^t|t-s|^{-\frac{p-1}{2}}\mathbf{E}\left[||(u_s-\bar{u}_s)
1_{\{s< R_K\}}||_p^p\right]ds.
\end{align*}
The generalized Gronwall Lemma (see, e.g., \cite[Theorem 1.2]{Lin:2013}) implies that 
\begin{align*}
\mathbf{E}\left[||(u_t-\bar{u}_t)
1_{\{t< R_K\}}||_p^p\right]\equiv0.
\end{align*}

Therefore, by the fact that the sample paths (on variable $t$) of $u_t,\bar{u}_t$ are right continuous, the pathwise uniqueness holds for all $t\in[0,R_K)$. Combining Lemmas \ref{lem:stopping} and \ref{lem:consistency}
the desired result follows by letting $K\uparrow\infty$.
Recalling (\ref{eq:solution}), it is easy to see that the moment estimate (\ref{eq:moment}) follows from (\ref{eq:truncated-moment}), which completes the proof.
\qed
\end{Tproof}

\begin{remark}
By Remark \ref{re:well-defined},  (\ref{eq:solution}) and 
Lemma \ref{lem:stopping},   
the (stochastic) integrals 
on the right-hand side of (\ref{eq:mildform})
are well defined.
\end{remark}

\section{Proof of Theorem \ref{th:comparison}}
\label{sec4}

Recalling (\ref{eq:solution}) and the proof of Theorem \ref{th:existence},
we know that the solution $u_t$ of equation (\ref{eq:originalequation1}) can be regarded as the limit of solutions $(u^K_t)_{K\geq1}$ to
equation (\ref{truncated-eq}). Therefore, to prove the
comparison theorem on
solutions of equation (\ref{eq:originalequation1}),
it suffices to prove the comparison theorem
on solutions $(u^K_t)_{K\geq1}$ of equation (\ref{truncated-eq}).
To this end, for any given $K\geq1$ and $T>0$,
similar to equation (\ref{truncated-eq}), we
consider the stochastic heat equation
\begin{align}
\label{truncated-eq2}
\left\{\begin{array}{lcl}
\dfrac{\partial v^K(t,x)}{\partial t}=\dfrac{1}{2}
\dfrac{\partial^2v^K(t,x)}{\partial x^2}
+g(t,x,v^K(t,x))
\\[0.3cm]
\quad\quad\quad\quad\quad\,
+\varphi(t-,x,v^K(t-,x))\dot{L}^K_{\alpha}(t,x),
&& (t,x)\in (0,T]
\times(0,L),\\[0.3cm]
v^K(0,x)=v_0(x),&&x\in[0,L],\\[0.3cm]
v^K(t,0)=v^K(t,L)=0,&& t\in[0,T].
\end{array}\right.
\end{align}
where the initial function $v_0$
satisfies $\mathbf{E}[||v_0||_H^2]<\infty$, and
the drift
coefficient $g:[0,T]\times[0,L]\times\mathbb{R}\rightarrow\mathbb{R}$
and
noise coefficient
$\varphi:[0,T]\times[0,L]\times\mathbb{R}\rightarrow\mathbb{R}$ also
satisfy  {\rm\textbf{Hypothesis}} \ref{hypotheses}.

Proposition \ref{prop:existence} implies that 
there exists a pathwise unique strong
$H$-valued c\`{a}dl\`{a}g solution
$v^K_t\equiv\{v^K(t,\cdot),t\in[0,T]\}$
to equation (\ref{truncated-eq2})
and
\begin{equation*}
\label{eq:truncated-moment2}
\sup_{0\leq t\leq T}\mathbf{E}\left[||v^K_t||_H^2\right]<\infty.
\end{equation*}

Given $K\geq1$,
we now present the comparison theorem on $H$-valued c\`{a}dl\`{a}g
solutions $u_t^K,v_t^K$ to equations
(\ref{truncated-eq}) and (\ref{truncated-eq2})  in the following proposition.

\begin{proposition}
\label{prop:comparison}
Suppose that {\rm\textbf{Hypothesis}} {\rm\ref{hypotheses}} holds
and  $c_{-}=0$ in (\ref{eq:jumpsizemeasure}).    
Given $K\geq1$, let
$u_t^K\equiv\{u^K(t,\cdot),t\in[0,T]\}$ and
$v_t^K\equiv\{v_t^K(t,\cdot),t\in[0,T]\}$ are the strong $H$-valued c\`{a}dl\`{a}g
solutions
to equations (\ref{truncated-eq})
and (\ref{truncated-eq2}), respectively.
If
$u_0(x)\leq v_0(x)\,\,\text{for a.e.}\,\,x\in[0,L]$
and
$f(t,x,y)\leq g(t,x,y)\,\,\text{for all}\,\,(t,x,y)\in[0,T]\times[0,L]\times\mathbb{R}$,
we have
\begin{align*}
\mathbf{P}\left[u_t^K\leq v_t^K\,\,\,\text{for all}\,\,t\in[0,T]\right]=1.
\end{align*}
\end{proposition}
\begin{proof}
For any $\epsilon>0$ and $x,y\in[0,L]$, we know that $G_{\epsilon}(x,y)$ 
is the fundamental solution of the 
heat equation (\ref{eq:fund-solu}) satisfying (\ref{eq:convolution}).
Since $u_t^K,v_t^K$ are two strong $H$-valued c\`{a}dl\`{a}g solutions of equations (\ref{truncated-eq})
and (\ref{truncated-eq2}), respectively, it holds by setting 
$G_{\epsilon}^x=G_{\epsilon}(x,\cdot)$ that
\begin{align}
\label{eq:eps-u}
\langle u^K_t,G_{\epsilon}^x\rangle
&=\langle u_0,G_{\epsilon}^x\rangle
+\dfrac{1}{2}\int_0^t \langle u^K_s,\dfrac{\partial^2}{\partial y^2}G_{\epsilon}^x\rangle ds
+\int_0^t \langle f(s,\cdot,u^K_s),G_{\epsilon}^x\rangle ds
\nonumber\\
&\quad+\int_0^{t+}\int_0^L
\int_{\mathbb{R}\setminus\{0\}}
\varphi(s-,y,u^K(s-,y))G_{\epsilon}(x,y)z\tilde{N}^K(ds,dy,dz)
\nonumber\\
&=\langle u_0,G_{\epsilon}^x\rangle
+\dfrac{1}{2}\int_0^t \langle u^K_s,\dfrac{\partial^2}{\partial x^2}G_{\epsilon}^x\rangle ds
+\int_0^t \langle f(s,\cdot,u^K_s),G_{\epsilon}^x\rangle ds
\nonumber\\
&\quad+\int_0^{t+}\int_0^L
\int_{\mathbb{R}\setminus\{0\}}
\varphi(s-,y,u^K(s-,y))G_{\epsilon}(x,y)z\tilde{N}^K(ds,dy,dz)
\nonumber\\
&=\langle u_0,G_{\epsilon}^x\rangle
+\dfrac{1}{2}
\int_0^t\dfrac{\partial^2}{\partial x^2} 
\langle u^K_s,G_{\epsilon}^x\rangle ds
+\int_0^t \langle f(s,\cdot,u^K_s),G_{\epsilon}^x\rangle ds
\nonumber\\
&\quad+\int_0^{t+}\int_0^L
\int_{\mathbb{R}\setminus\{0\}}
\varphi(s-,y,u^K(s-,y))G_{\epsilon}(x,y)z\tilde{N}^K(ds,dy,dz),
\end{align}
where the second equality follows from the symmetry on the variables
$x,y$ of the function $G_{\epsilon}(x,y)$. Similarly, 
\begin{align}
\label{eq:eps-v}
\langle v^K_t,G_{\epsilon}^x\rangle
&=\langle v_0,G_{\epsilon}^x\rangle
+\dfrac{1}{2}\int_0^t\dfrac{\partial^2}{\partial x^2}
\langle v^K_s,G_{\epsilon}^x\rangle ds+\int_0^t 
\langle g(s,\cdot,v^K_s),G_{\epsilon}^x\rangle
ds
\nonumber\\
&\quad+\int_0^{t+}\int_0^L
\int_{\mathbb{R}\setminus\{0\}}
\varphi(s-,y,v^K(s-,y))G_{\epsilon}(x,y)z\tilde{N}^K(ds,dy,dz).
\end{align}

Setting
$w_{t}^K=u^K_t-v^K_t$
and 
$w_{t}^{K,\epsilon}(x)=\langle w_t^K,G_{\epsilon}^x\rangle$, then it follows
from (\ref{eq:eps-u})-(\ref{eq:eps-v}) that
\begin{align}
\label{eq:semi-mart}
w_{t}^{K,\epsilon}(x)&=w_{0}^{K,\epsilon}(x)
+\dfrac{1}{2}\int_0^t\dfrac{\partial^2}{\partial x^2} w_{s}^{K,\epsilon}(x)ds
+\int_0^t[F_{s}^{K,\epsilon,1}(x)+F_{s}^{K,\epsilon,2}(x)]ds
\nonumber\\
&\quad+\int_0^{t+}\int_0^L
\int_{\mathbb{R}\setminus\{0\}}H_{s-}^{K,\epsilon}(x,y,z)\tilde{N}^K(ds,dy,dz),
\end{align}
where 
\begin{align*}
&F_{s}^{K,\epsilon,1}(x):=
\langle f(s,\cdot,u^K_s)-f(s,\cdot,v^K_s),G_{\epsilon}^x\rangle,
\\
&F_{s}^{K,\epsilon,2}(x):=
\langle f(s,\cdot,v^K_s)-g(s,\cdot,v^K_s),G_{\epsilon}^x\rangle,
\\
&H_{s}^{K,\epsilon}(x,y,z):=[\varphi(s,y,u^K(s,y))-\varphi(s,y,v^K(s,y))]G_{\epsilon}(x,y)z.
\end{align*}

For any $\eta>0$,
let
$\Psi_{\eta}:\mathbb{R}\rightarrow\mathbb{R}$
and $\psi_{\eta}:\mathbb{R}\rightarrow\mathbb{R}$ be
two functions defined by
\begin{equation*}
\Psi_{\eta}(x):=
\left\{\begin{array}{lcl}
0,&&x\leq0,
\\[0.4cm]
2{\eta}x,&&x\in[0,\frac{1}{{\eta}}],
\\[0.4cm]
2,&&x\geq\frac{1}{\eta},
\end{array}\right.
\end{equation*}
and
\begin{align*}
\psi_{\eta}(x):=1_{\{x\geq0\}}\int_0^xdy\int_0^ydz
\Psi_{\eta}(z),
\end{align*}
respectively. Then it follows that
\begin{enumerate}[(a)]
\item \label{eq:eta1}$\psi_{\eta}$ is a $C^2$-function on
$\mathbb{R}$;
\item \label{eq:eta2}
$
0\leq\psi_{\eta}^{'}(x)\leq 2x^{+}=
2\max\{{x,0}\},\,\,0\leq\psi_{\eta}^{''}(x)\leq 2,\, \text{for all}\,\,
x\in\mathbb{R}$;

\item \label{eq:eta3}
$
\lim_{\eta\rightarrow\infty}\psi_{\eta}(x)
=(x^{+})^2,\,\,\lim_{\eta\rightarrow\infty}\psi_{\eta}^{'}(x)=2x^{+},\,\,
0\leq\lim_{\eta\rightarrow\infty}\psi_{\eta}^{''}(x)\leq 2,\,\text{for all}\,
x\in\mathbb{R}.
$
\end{enumerate}

Since $w^{K,\epsilon}_t(x)$ given by (\ref{eq:semi-mart})
is a real valued semi-martingale
for each fixed $x\in[0,L]$, then by (\ref{eq:eta1}) in the above and It\^{o}'s formula (see, e.g., Applebaum \cite[Theorem 4.4.7]{Applebaum:2009}), 
\begin{align*}
\psi_{\eta}(w^{K,\epsilon}_t(x))&=\psi_{\eta}(w^{K,\epsilon}_0(x))
+\dfrac{1}{2}\int_0^t \psi_{\eta}^{'}(w^{K,\epsilon}_s(x))
\dfrac{\partial^2}{\partial x^2}w^{K,\epsilon}_s(x) ds
\\
&\quad+\int_0^t \psi_{\eta}^{'}(w^{K,\epsilon}_s(x))[F_{s}^{K,\epsilon,1}(x)+F_{s}^{K,\epsilon,2}(x)]ds
\\
&\quad+\int_0^{t+}\int_0^L
\int_{\mathbb{R}\setminus\{0\}}\{\psi_{\eta}(w^{K,\epsilon}_{s-}(x)
+H_{s-}^{K,\epsilon}(x,y,z))-\psi_{\eta}(w^{K,\epsilon}_{s-}(x))\}
\tilde{N}^K(ds,dy,dz)
\\
&\quad+\int_0^{t}\int_0^L
\int_{\mathbb{R}\setminus\{0\}}\{\psi_{\eta}(w^{K,\epsilon}_s(x)+H_{s}^{K,\epsilon}(x,y,z))-\psi_{\eta}(w^{K,\epsilon}_s(x))
\\
&\quad\quad\quad\quad\quad\quad\quad\quad
\quad\quad\quad\quad\quad\,-\psi_{\eta}^{'}(w^{K,\epsilon}_s(x))
H_{s}^{K,\epsilon}(x,y,z)\}dsdy\nu_{\alpha}^K(dz).
\end{align*}
By stochastic Fubini's theorem, we get
\begin{align}
\label{eq:estimate}
&\mathbf{E}\left[\int_0^L\psi_{\eta}(w^{K,\epsilon}_t(x))dx\right]
\nonumber\\
&=\mathbf{E}\left[\int_0^L\psi_{\eta}(w^{K,\epsilon}_0(x))dx\right]
+\dfrac{1}{2}\mathbf{E}\left[\int_0^t\int_0^L\psi_{\eta}^{'}(w^{K,\epsilon}_s(x))\dfrac{\partial^2}{\partial x^2}w_{s}^{K,\epsilon}(x)dxds\right]
\nonumber\\
&\quad+\mathbf{E}\left[\int_0^t\int_0^L\psi_{\eta}^{'}(w^{K,\epsilon}_s(x))[F_{s}^{K,\epsilon,1}(x)+F_{s}^{K,\epsilon,2}(x)]dxds\right]
\nonumber\\
&\quad+\mathbf{E}\Bigg[\int_0^{t}\int_0^L
\int_{\mathbb{R}\setminus\{0\}}\Bigg\{
\int_0^L[\psi_{\eta}(w^{K,\epsilon}_s(x)+H_{s}^{K,\epsilon}(x,y,z))-\psi_{\eta}(w^{K,\epsilon}_s(x))
\nonumber\\
&\quad\quad\quad\quad\quad\quad\quad\quad\quad\quad\quad\quad\,\,\,
-\psi_{\eta}^{'}(w^{K,\epsilon}_s(x))
H_{s}^{K,\epsilon}(x,y,z)]dy\Bigg\}dsdx\nu_{\alpha}^K(dz)\Bigg].
\end{align}

The assumption  $u_0(x)\leq v_0(x)$ 
 {for a.e.} $x\in[0,L]$
implies that
\begin{align}
\label{eq:estimate1}
\int_0^L\psi_{\eta}(w^{K,\epsilon}_0(x))dx=0.
\end{align}

Since $\psi_{\eta}^{'}(w^{K,\epsilon}_s(0))=\psi_{\eta}^{'}(w^{K,\epsilon}_s(L))=0$, it holds by the integration-by-parts formula 
and (\ref{eq:eta3}) that
\begin{align}
\label{eq:estimate2}
\int_0^L\psi_{\eta}^{'}(w^{K,\epsilon}_s(x))\dfrac{\partial^2}{\partial x^2}w_{s}^{K,\epsilon}(x)dx=-\int_0^L\psi_{\eta}^{''}(w^{K,\epsilon}_s(x))
\left|\dfrac{\partial}{\partial x}w_{s}^{K,\epsilon}(x)\right|^2dx\leq0.
\end{align}

The assumption on
$f(t,x,w)\leq g(t,x,w)\,\,\text{for all}\,\,(t,x,w)\in[0,T]\times[0,L]\times\mathbb{R}$ and (\ref{eq:eta2}) imply that
\begin{align}
\label{eq:estimate3}
\int_0^L\psi_{\eta}^{'}(w^{K,\epsilon}_s(x))F_{s}^{K,\epsilon,2}(x)dx\leq0.
\end{align}

Therefore, by substituting (\ref{eq:estimate1})-(\ref{eq:estimate3}) into
(\ref{eq:estimate}), we obtain
\begin{align*}
&\mathbf{E}\left[\int_0^L\psi_{\eta}(w^{K,\epsilon}_t(x))dx\right]
\nonumber\\
&\leq\mathbf{E}\left[\int_0^t\int_0^L\psi_{\eta}^{'}(w^{K,\epsilon}_s(x))F_{s}^{K,\epsilon,1}(x)dxds\right]
\\
&\quad+\mathbf{E}\Bigg[\int_0^{t}\int_0^L
\int_{\mathbb{R}\setminus\{0\}}\Bigg\{
\int_0^L[\psi_{\eta}(w^{K,\epsilon}_s(x)+H_{s}^{K,\epsilon}(x,y,z))-\psi_{\eta}(w^{K,\epsilon}_s(x))
\nonumber\\
&\quad\quad\quad\quad\quad\quad\quad\quad\quad\quad\quad\quad\,\,\,
-\psi_{\eta}^{'}(w^{K,\epsilon}_s(x))
H_{s}^{K,\epsilon}(x,y,z)]dy\Bigg\}dsdx\nu_{\alpha}^K(dz)\Bigg].
\end{align*}

Letting $\eta\uparrow\infty$, it holds by (\ref{eq:eta3}) and
{\rm\textbf{Hypothesis}} \ref{hypotheses}
($f$ is Lipschitz continuous) that
\begin{align*}
&\mathbf{E}\left[\int_0^L|(w^{K,\epsilon}_t(x))^{+}|^2dx\right]
\\
&\leq C\mathbf{E}\left[\int_0^t\int_0^L(w_s^{K,\epsilon}(x))^{+}
|w_s^{K,\epsilon}(x)|dxds\right]
\\
&\quad+\mathbf{E}\Bigg[\int_0^{t}\int_0^L
\int_{\mathbb{R}\setminus\{0\}}\Bigg\{
\int_0^L
[|(w^{K,\epsilon}_s(x)+H_{s}^{K,\epsilon}(x,y,z))^{+}|^2-|(w^{K,\epsilon}_s(x))^{+}|^2
\nonumber\\
&\quad\quad\quad\quad\quad\quad\quad\quad\quad\quad\quad\quad\,\,\,
-2(w^{K,\epsilon}_s(x))^{+}
H_{s}^{K,\epsilon}(x,y,z)]dy\Bigg\}dsdx\nu_{\alpha}^K(dz)\Bigg].
\end{align*}
By {\rm\textbf{Hypothesis}} \ref{hypotheses}
($\varphi$ is non-decreasing) and assumption
$c_{-}=0$ in (\ref{eq:jumpsizemeasure}),
\begin{align}
\label{eq:crucial-equality}
&|(w^{K,\epsilon}_s(x)+H_{s}^{K,\epsilon}(x,y,z))^{+}|^2
-|(w^{K,\epsilon}_s(x))^{+}|^2
-2(w^{K,\epsilon}_s(x))^{+}
H_{s}^{K,\epsilon}(x,y,z)
\nonumber\\
&=1_{\{w_s^{K,\epsilon}(x)\geq0\}}|H_s^{K,\epsilon}(x,y,z)|^2.
\end{align}
(Note that the above equality
also holds under {\rm\textbf{Hypothesis}} \ref{hypotheses} 
for which $\varphi$ is non-increasing and assumption
$c_{+}=0$ in (\ref{eq:jumpsizemeasure}).)

Recalling (\ref{eq:convolution}) and
by letting $\epsilon\downarrow0$ and using Lebesgue's dominated 
convergence theorem, we obtain
\begin{align*}
\mathbf{E}\left[\int_0^L|(w^{K}_t(x))^{+}|^2dx\right]
&\leq C\mathbf{E}\left[\int_0^t\int_0^L(w_s^K(x))^{+}
|w_s^K(x)|dxds\right]
\\
&\quad+\mathbf{E}\Bigg[\int_0^{t}\int_0^L
\int_{\mathbb{R}\setminus\{0\}}1_{\{w^K_s(x)\geq0\}}|H_s^{K}(x,z)|^2
dsdx\nu_{\alpha}^K(dz)\Bigg].
\end{align*}
where
\begin{align*}
w_t^K(x)&=u^K(t,x)-v^K(t,x),
\\
F_{s}^{K}(x)&=f(s,x,u^K(s,x))-f(s,x,v^K(s,x)),
\\
H_{s}^{K}(x,z)&=[\varphi(s,x,u^K(s,x))-\varphi(s,x,v^K(s,x))]z.
\end{align*}
By {\rm\textbf{Hypothesis}} \ref{hypotheses}
($\varphi$ is Lipschitz continuous) and  
(\ref{eq:jumpestimate}) with $p=2$, 
\begin{align*}
\mathbf{E}\left[\int_0^L|(w^{K}_t(x))^{+}|^2dx\right]
&\leq C\mathbf{E}\left[\int_0^t\int_0^L(w_s^K(x))^{+}
|w_s^K(x)|dxds\right]
\\
&\quad+C\left(\int_{\mathbb{R}\setminus\{0\}}z^2\nu_{\alpha}^K(dz)\right)
\mathbf{E}\Bigg[\int_0^{t}\int_0^L1_{\{w_s^K(x)\geq0\}}
|w_s^K(x)|^2dsdx\Bigg]
\\
&\leq C_{\alpha,K}\mathbf{E}\Bigg[\int_0^{t}\int_0^L
|(w_t^K(x))^{+}|^2dxds\Bigg],
\end{align*}
that is,
\begin{align*}
\mathbf{E}\left[||(w^{K}_t)^{+}||_{H}^2\right]\leq C_{\alpha,K}
\mathbf{E}\Bigg[\int_0^{t}||(w_s^K)^{+}||_H^2ds\Bigg].
\end{align*}
Applying the Gronwall lemma, one can conclude that
\begin{align*}
\sup_{0\leq t\leq T}
\mathbf{E}[||(w_t^K)^{+}||_H^2]=0.
\end{align*}
By the right continuity of sample paths (on variable $t$) of $u^K_t$ 
and $v^K_t$, we obtain
\begin{align*}
\mathbf{P}\left[w_t^K\leq 0\,\,\,\text{for all}\,\,t\in[0,T]\right]=1,
\end{align*}
that is,
\begin{align*}
\mathbf{P}\left[u^K_t\leq v^K_t\,\,\,\text{for all}\,\,t\in[0,T]\right]=1,
\end{align*}
which completes the proof.
\qed
\end{proof}

\begin{Tproof}~\textbf{of Theorem \ref{th:comparison}.}
Recalling Lemma \ref{lem:stopping} and
(\ref{eq:solution}), the proof is finished 
by Proposition \ref{prop:comparison} and 
by letting $K\uparrow\infty$.
\qed
\end{Tproof}

At the end of this section we outline another
approach based on infinite dimensional
stochastic differential equations theory
to prove Theorem \ref{th:comparison}. This kind of method
first appears in \cite{Donati:1993} in showing a comparison principle for stochastic heat equations driven by Gaussian space-time white noises.
The key of this approach is to project the truncated $\alpha$-stable white noise $\dot{L}_{\alpha}^K$ given in equation (\ref{truncated-eq})
onto a finite dimensional spatial space to construct a
sequence of evolution equations having semi-martingale form so that
the It\^{o} formula 
for infinite dimensional semi-martingales 
can be applied in showing the
comparison principle for the evolution equations. 

We now construct a sequence of evolution equations.
Let $\{h_n\}_{n\geq1}$ be an orthonormal basis
of the Hilbert space $H$ such that
$h_n\in L^{\infty}([0,L])$ for all $n\geq1$, that is,
\begin{equation}
\label{eq:h_n}
\sup_{x\in[0,L]}|h_n(x)|\leq C,\,\, \text{for all}\,\,n\geq1.
\end{equation}
Given $n,K\geq1$ and $T>0$, let $L_n^K\equiv\{L_n^K(t),t\in[0,T]\}$ be a real-valued  process defined by
\begin{align}
\label{eq:LnK}
L_n^K(t):=
\int_0^{t}\int_0^L\int_{\mathbb{R}\setminus\{0\}}h_n(x)z
\tilde{N}^K(ds,dx,dz),\,\, t\in[0,T].
\end{align}
Then $(L_n^K)_{n\geq1}$ is a family of mutually
uncorrelated real-valued truncated $\alpha$-stable
processes. 
Given $m,K\geq1$ and $T>0$, let $Z_m^K\equiv\{Z_m^K(t),t\in[0,T]\}$ be a $H$-valued truncated $\alpha$-stable process defined
by
\begin{align}
\label{eq:ZmK}
Z_m^K(t):=\sum_{n=1}^mL_n^K(t)h_n,\,\, t\in[0,T].
\end{align}
Then we consider a
evolution equation of the form
\begin{equation}
\label{eq:evolutionequation}
\left\{\begin{array}{lcl}
du_m^K(t)=Au_m^K(t)dt+f(t,\cdot,u_m^K(t))dt+
\Sigma(t-,\cdot,u_m^K(t-))dZ_m^K(t),&& t\in (0,T],
\\[0.3cm]
u_m^K(0,x)=u_0(x),&&x\in [0,L],
\\[0.3cm]
u_m^K(t,0)=u_m^K(t,L)=0,&& t\in[0,T],
\end{array}\right.
\end{equation}
where $u_m^K(t)\equiv \{u_m^K(t,\cdot),t\in[0,T]\}$,
$A=\frac{1}{2}\frac{\partial^2}{\partial x^2}$, 
and
$\Sigma$ is the operator of multiplication by
$\varphi$, that is, for any $t\in[0,T]$ and $w\in H$,
$\Sigma(t,\cdot,w):L^{\infty}([0,L])\rightarrow H$
is defined by
$
\Sigma(t,\cdot,w)(h):=\varphi(t,\cdot,w)h,
h\in L^{\infty}([0,L]).
$

By Gy\"{o}ngy 
\cite[Theorems 2.9-2.10]{Gyongy:1982}, 
if the initial function $u_0$ satisfies
$\mathbf{E}[||u_0||_H^2]<\infty$, then under 
 {\rm\textbf{Hypothesis}} {\rm\ref{hypotheses}} {\rm(i)}
there exists a pathwise unique strong
$H^1_0$-valued c\`{a}dl\`{a}g solution
$u_m^K(t)\equiv\{u_m^K(t,\cdot),t\in[0,T]\}$
to equation (\ref{eq:evolutionequation})
satisfying
\begin{align*}
u_m^K(t,\cdot)&=u_0(\cdot)+\int_0^tAu_m^K(s,\cdot)ds
+\int_0^tf(s,\cdot,u_m^K(s,\cdot))ds
+\int_0^{t+}\Sigma(s-,\cdot,u_m^K(s-,\cdot))dZ_m^K(s),
\end{align*}
where $H^1_0$ is the closure of $C^{\infty}([0,L])$ in $H^1$, and 
$H^1$ is the usual Sobolev space consists of all functions 
in $H$ for which the first order weak (generalized) 
derivative belongs to $H$.

Replacing $u_0$ and  $f$ in equation (\ref{eq:evolutionequation}) 
by $v_0$ and $g$,
respectively, then  there exists a pathwise unique
strong $H^1_0$-valued c\`{a}dl\`{a}g solution
$v_m^K\equiv\{v_m^K(t,\cdot),t\in[0,T]\}$ 
satisfying
\begin{align*}
v_m^K(t,\cdot)&=v_0(\cdot)+\int_0^tAv_m^K(s,\cdot)ds
+\int_0^t g(s,\cdot,v_m^K(s,\cdot))ds
+\int_0^{t+}\Sigma(s-,\cdot,v_m^K(s-,\cdot))dZ_m^K(s).
\end{align*}

Let $\Phi_{\eta}:H\rightarrow\mathbb{R}$ be an
operator defined by
$
\Phi_{\eta}(h):=\int_{0}^L\psi_{\eta}(h(x))dx, h\in H.
$
Then, by (\ref{eq:eta1}), 
$\Phi_{\eta}$ is twice Fr\'{e}chet differentiable.
Set $w_m^K(t,\cdot):=u_m^K(t,\cdot)-v_m^K(t,\cdot)$. 
By It\^{o}'s formula for infinite dimensional
semi-martingale
(see, e.g.,\cite[Theorem 3.7.2]{Mandrekar:2015}), one can show that
\begin{align}
\label{eq:ito}
\Phi_{\eta}(w_m^K(t,\cdot))
&=
\Phi_{\eta}(w_m^K(0,\cdot))
+
\int_0^t\langle\psi_{\eta}^{'}(w_m^K(s,\cdot)),Aw_m^K(s,\cdot)
\rangle ds
\nonumber\\
&\quad+
\int_0^t(\psi_{\eta}^{'}(w_m^K(s,\cdot)),f(s,\cdot,u_m^K(s,\cdot))-g(s,\cdot,v_m^K(s,\cdot)))ds
\nonumber\\
&\quad+\int_0^{t+}\int_0^L
\int_{\mathbb{R}\setminus\{0\}}\{
\Phi_{\eta}(w_m^K(s-,\cdot)+G_m^K(s-,\cdot,y,z))
\nonumber\\
&\quad\quad\quad\quad\quad\quad
\quad\quad\quad
-
\Phi_{\eta}(w_m^K(s-,\cdot))\}\tilde{N}^K(ds,dy,dz)
\nonumber\\
&\quad+\int_0^{t}\int_0^L
\int_{\mathbb{R}\setminus\{0\}}\{
\Phi_{\eta}(w_m^K(s,\cdot)+G_m^K(s,\cdot,y,z))-
\Phi_{\eta}(w_m^K(s,\cdot))
\nonumber\\
&\quad\quad\quad\quad\quad\quad
\quad\quad\quad
-(\psi_{\eta}^{'}(w_m^K(s,\cdot)),
G_m^K(s,\cdot,y,z))\}dsdy\nu_{\alpha}^K(dz).
\end{align}
Here $(\cdot,\cdot)$ is the inner product of $H$, 
$\langle\cdot,\cdot\rangle$ is the duality product between $H^{-1}$ (dual space of $H^1_0$) and $H^1_0$, and 
$$G_m^K(s,\cdot,y,z):=\sum_{n=1}^m[\varphi(s,\cdot,u_m^K(s,\cdot))-
\varphi(s,\cdot,v_m^K(s,\cdot))]h_n(\cdot)h_n(y)z.$$

Similar to the estimates 
in the proof of Theorem \ref{th:comparison}, one can get
\begin{align*}
\mathbf{E}[||(w_m^K(t,\cdot))^{+}||_H^2]
&\leq C_{m,\alpha,K}\int_0^t\mathbf{E}[||(w_m^K(s,\cdot))^{+}||_H^2]ds.
\end{align*}
The Gronwall lemma and the right continuity of the sample path of $w_m^K$ imply that  
\begin{align}
\label{eq:com}
\mathbf{P}\left[u_m^K(t,\cdot)\leq v_m^K(t,\cdot)\,\,\,\text{for all}\,\,t\in[0,T]\right]=1.
\end{align}

We now show that for any given $T>0$
\begin{align}
\label{eq:key-equality}
&\lim_{m\rightarrow\infty}\sup_{0\leq t\leq T}
\mathbf{E}[||u^K(t,\cdot)-u_m^K(t,\cdot)||_H^2]=0,
\\
&\label{eq:key-equality2}
\lim_{m\rightarrow\infty}\sup_{0\leq t\leq T}
\mathbf{E}[||v^K(t,\cdot)-v_m^K(t,\cdot)||_H^2]=0,
\end{align}
where $u^K(t,\cdot)$ and $v^K(t,\cdot)$ are the strong
$H$-valued c\`{a}dl\`{a}g
solutions of equations
(\ref{truncated-eq}) and (\ref{truncated-eq2}).
We just prove (\ref{eq:key-equality}) since  (\ref{eq:key-equality2}) is similar. By (\ref{eq:LnK}) and (\ref{eq:ZmK}),
the mild form of equation 
(\ref{eq:evolutionequation}) is
\begin{align*}
u_m^K(t,x)&=
\int_0^LG_t(x,y)u_0(y)dy
+\int_0^{t}\int_0^LG_{t-s}(x,y)
f(s,y,u_m^K(s,y))dsdy
\nonumber\\
&\quad+\sum_{n=1}^{m}
\int_0^{t+}\int_0^L\int_{\mathbb{R}\setminus\{0\}} \int_0^LG_{t-s}(x,y)
\varphi(s-,y,u_m^K(s-,y))
h_n(y)dyh_n(w)z\tilde{N}^K(ds,dw,dz)
\end{align*}
for all $t\in [0,\infty)$ and a.e. $x\in [0, L]$. 
By rewriting (\ref{eq:mild-truncted}) with the form as
\begin{align*}
u^K(t,x)&=\int_0^LG_t(x,y)u_0(y)dy+\int_0^t\int_0^LG_{t-s}(x,y)f(s,y,u^K(s,y))dsdy
\\
&\quad+\sum_{n=1}^{m}
\int_0^{t+}\int_0^L\int_{\mathbb{R}\setminus\{0\}} \int_0^LG_{t-s}(x,y)
\varphi(s-,y,u^K(s-,y))
h_n(y)dyh_n(w)z\tilde{N}^K(ds,dw,dz)
\nonumber\\
&\quad+\int_0^{t+}\int_0^L
\int_{\mathbb{R}\setminus\{0\}}G_{t-s}(x,y)
\varphi(s-,y,u^K(s-,y))z\tilde{N}^K(ds,dy,dz)
\\
&\quad-\sum_{n=1}^{m}
\int_0^{t+}\int_0^L\int_{\mathbb{R}\setminus\{0\}} \int_0^LG_{t-s}(x,y)
\varphi(s-,y,u^K(s-,y))
h_n(y)dyh_n(w)z\tilde{N}^K(ds,dw,dz),
\end{align*} 
we have
\begin{align*}
u^K(t,x)-u_m^K(t,x)=A_m^K(t,x)+B_m^K(t,x)+C_m^K(t,x),
\end{align*}
where
\begin{align*}
A_m^K(t,x)&=\int_0^t\int_0^LG_{t-s}(x,y)
(f(s,y,u^K(s,y)-f(s,y,u_m^K(s,y))dsdy,
\\
B_m^K(t,x)&=\sum_{n=1}^{m}
\int_0^{t+}\int_0^L\int_{\mathbb{R}\setminus\{0\}}
\int_0^L G_{t-s}(x,y)
\nonumber\\
&\quad\quad\quad\quad\quad\times
[\varphi(s-,y,u^K(s-,y))-\varphi(s-,y,u_m^K(s-,y))]
h_n(y)dyh_n(w)z\tilde{N}^K(ds,dw,dz),
\\
C_m^K(t,x)&=\int_0^{t+}\int_0^L\int_{\mathbb{R}\setminus\{0\}}
(\xi_{t,x}^{K}(s-,y)-
\xi_{t,x}^{m,K}(s-,y))z\tilde{N}^K(ds,dy,dz),
\end{align*}
in which
\begin{align}
\label{eq:Psik}
\xi_{t,x}^{K}(s,y)&=G_{t-s}(x,y)
\varphi(s,y,u^K(s,y)),
\end{align}
\begin{align}
\label{eq:Psimk}
\xi_{t,x}^{m,K}(s,y)&=\sum_{n=0}^{m}
\left(\int_0^L
G_{t-s}(x,r)\varphi(s,r,u^K(s,r))h_n(r)dr\right)
h_n(y).
\end{align}

By (\ref{eq:Psik})-(\ref{eq:Psimk}), 
\begin{align*}
\xi_{t,x}^{m,K}(s,y)=\sum_{n=1}^m
(\xi_{t,x}^{K}(s,\cdot),h_n(\cdot))h_n(y).
\end{align*}
For each fixed $K\geq1,s\leq t\in[0,T]$
and $x\in[0,L]$,
it holds by {\rm Hypothesis} {\rm\ref{hypotheses} \rm(i)} and (\ref{eq:truncated-moment}) for $p=2$
that
\begin{align*}
||\xi_{t,x}^{K}(s,\cdot)||_H^2
&\leq C\int_0^L|G_{t-s}(x,y)|^2(1+|u^K(s,y))|^2)dy
\\
&\leq C\sup_{y\in[0,L]}|G_{t-s}(x,y)|^2(L+\sup_{s\in[0,T]}||u^K(s,\cdot)||_H^2)<\infty.
\end{align*}
Therefore, for each fixed $K\geq1,s\leq t\in[0,T]$
and $x\in[0,L]$
\begin{align}
\label{eq:Psi1}
||\xi_{t,x}^{m,K}(s,\cdot)||_H\leq ||\xi_{t,x}^{K}(s,\cdot)||_H
\,\,\text{and}\,\,
\lim_{m\rightarrow\infty}||\xi_{t,x}^{K}(s,
\cdot)-\xi_{t,x}^{m,K}(s,\cdot)||_H=0.
\end{align}

We now respectively estimate $A_m^K(t,x), B_m^K(t,x)$ and $C_m^K(t,x)$. For $A_m^K(t,x)$, it follows from the H\"{o}lder inequality and (\ref{eq:Greenetimation2}) that
\begin{align}
\label{eq:AmK}
\mathbf{E}[||A_m^K(t,\cdot)||_H^2]
\leq C\mathbf{E}\left[\int_0^t\int_0^L|t-s|^{-\frac{1}{2}}
|f(s,y,u^K(s,y)-f(s,y,u_m^K(s,y))|^2dsdy\right].
\end{align}

For $B_m^K(t,x)$, it follows from the uncorrelatedness of
$(L_n^K)_{n\geq1}$, the
Burkholder-Davis-Gundy  inequality and (\ref{eq:jumpestimate}) with $p=2$ that
\begin{align*}
\mathbf{E}[||B_m^K(t,\cdot)||_H^2]
&\leq C_{\alpha,K}\int_0^L\mathbf{E}\Bigg[
\int_0^{t+}\int_0^L
\sum_{n=0}^{m}\Bigg(\int_0^LG_{t-s}(x,y)
[\varphi(s-,y,u^K(s-,y))
\\
&\quad\quad\quad-\varphi(s-,y,u_m^K(s-,y))]
h_n(y)dy\Bigg)^2h_n(w)^2dsdw
\Bigg]dx.
\end{align*}
Note that
\begin{align*}
&\sum_{n=0}^{m}\Bigg(\int_0^LG_{t-s}(x,y)
[\varphi(s-,y,u^K(s-,y))-\varphi(s-,y,u_m^K(s-,y))]
h_n(y)dy\Bigg)^2
\\
&\leq||G_{t-s}(x,\cdot)
[\varphi(s-,\cdot,u^K(s-,\cdot))
-\varphi(s-,\cdot,u_m^K(s-,\cdot))]||_H^2.
\end{align*}
Then by (\ref{eq:h_n}) and (\ref{eq:Greenetimation2}), one can show that
\begin{align}
\label{eq:BmK}
\mathbf{E}[||B_m^K(t,\cdot)||_H^2]
&\leq C_{\alpha,K}
\int_0^L
\mathbf{E}\Bigg[\int_0^{t}||G_{t-s}(x,\cdot)
[\varphi(s,\cdot,u^K(s,\cdot))
-\varphi(s,\cdot,u_m^K(s,\cdot))]||_H^2
ds\Bigg]dx
\nonumber\\
&\leq C_{\alpha,K}\mathbf{E}\Bigg[\int_0^{t}
\int_0^L|t-s|^{-\frac{1}{2}}
|\varphi(s,y,u^K(s,y))-\varphi(s,y,u_m^K(s,y))|^2
dyds\Bigg].
\end{align}

Given $m\geq1$ and $K\geq1$, let us define
$F_m^K(t):=
\mathbf{E}\left[||u^K(t,\cdot)-u_m^K(t,
\cdot)||_H^2\right],\,\,t\in[0,T].$
By (\ref{eq:AmK})-(\ref{eq:BmK}) and
 {\rm Hypothesis} {\rm\ref{hypotheses} \rm(i)}, we have
\begin{align}
\label{eq:ABmK}
\mathbf{E}[||A_m^K(t,\cdot)+B_m^K(t,\cdot)||_H^2]
\leq C_{p,T,\alpha,K}\int_0^{t}|t-s|^{-\frac{1}{2}}
F_m^K(s)ds.
\end{align}

For $C_m^K(t,x)$, the
Burkholder-Davis-Gundy inequality and  (\ref{eq:element-inequ})-(\ref{eq:jumpestimate}) together imply that
\begin{align*}
\mathbf{E}[||C_m^K(t,\cdot)||_H^2]
\leq C_{\alpha,K} \int_0^L\mathbf{E}\Bigg[
\int_0^{t}||\xi_{t,x}^{K}(s,\cdot)-
\xi_{t,x}^{m,K}(s,\cdot)||_H^2ds\Bigg]dx.
\end{align*}


By (\ref{eq:Psi1}), {\rm Hypothesis} {\rm\ref{hypotheses} \rm(i)} and
(\ref{eq:truncated-moment}) for $p=2$,
we obtain
\begin{align*}
\mathbf{E}[||C_m^K(t,\cdot)||_H^2]
&\leq C_{\alpha,K}\int_0^L\mathbf{E}\left[\int_0^t||\xi_{t,x}^{K}(s,
\cdot)||_H^2ds\right]dx
\\
&\leq C_{\alpha,K}\int_0^T|T-s|^{-\frac{1}{2}}ds
\left(1+\sup_{0\leq t\leq T }\mathbf{E}\left[||u^K(t,\cdot)||_H^2\right]\right)<\infty.
\end{align*}
Therefore, 
Lebesgue's dominated convergence theorem implies that
$\lim_{m\rightarrow\infty}
\mathbf{E}[||C_m^K(t,\cdot)||_H^2]=0.$
Combine the estimate (\ref{eq:ABmK}), for arbitrary 
$\epsilon>0$, there exists a constant $M$
such that for all $m\geq M$ we have
\begin{align*}
F_m^K(t)\leq C_{T,\alpha,K}\int_0^{t}|t-s|^{-\frac{1}{2}}
F_m^K(s)ds+C\epsilon,\,\, t\in[0,T].
\end{align*}
By the generalized Gronwall lemma 
(see, e.g., \cite[Theorem 1.2]{Lin:2013}) 
and by letting $m\uparrow\infty,\epsilon\downarrow0$, we obtain
\begin{align*}
\lim_{m\rightarrow\infty}\sup_{0\leq t\leq T }
\mathbf{E}[||u^K(t,\cdot)-u_m^K(t,\cdot)||_H^2]
=
\lim_{m\rightarrow\infty}\sup_{0\leq t\leq T }
F_m^K(t)=0.
\end{align*}

Therefore, for each $K\geq1$ and $t\in[0,\infty)$,
there exists a sub-sequence
$(m_k)_{k\geq1}$ such that
\begin{align*}
\mathbf{P}\left[u^K(t,\cdot)=
\lim_{k\rightarrow\infty}u_{m_k}^K(t,\cdot)\right]=
\mathbf{P}\left[v^K(t,\cdot)=
\lim_{k\rightarrow\infty}v_{m_k}^K(t,\cdot)\right]=1.
\end{align*}
Hence, 
by (\ref{eq:com})-(\ref{eq:key-equality2}) and the right continuity of sample paths of $u^K$ and $v^K$,
we obtain
\begin{align*}
\mathbf{P}\left[u^K(t,\cdot)\leq v^K(t,\cdot)\,\,\,\text{for all}\,\,t\geq0\right]=1.
\end{align*}
Recalling (\ref{eq:solution}), the proof is finished by letting $K\uparrow+\infty$.

\noindent \textbf{Acknowledgements}~~
This work is supported by the National Natural Science Foundation of China (NSFC) (Nos. 11631004, 71532001) and Natural Sciences and Engineering Research Council of Canada (RGPIN-2021-04100).

\end{document}